\documentclass[a4paper,11pt]{article}
\usepackage[latin1]{inputenc}
\usepackage[english]{babel}
\usepackage{amsmath}
\usepackage{amsfonts}
\usepackage{amssymb}
\usepackage{epsfig}
\usepackage{amsopn}
\usepackage{amsthm}
\usepackage{color}
\usepackage{graphicx}
\usepackage{enumerate}
\usepackage{mathrsfs}
\usepackage{cite}
\newtheorem{theorem}{Theorem}[section]

\newtheorem{lemma}[theorem]{Lemma}

\newtheorem*{theorem*}{Theorem}
\newtheorem*{lemma*}{Lemma}
\newtheorem*{remark*}{Remark}
\newtheorem*{definition*}{Definition}
\newtheorem*{proposition*}{Proposition}
\newtheorem*{corollary*}{Corollary}
\numberwithin{equation}{section}
\newcommand{\real}{\mathbb{R}}

\let\ced=\c         % cedilla
\def\qed{\,\unskip\kern 6pt \penalty 500
\raise -2pt\hbox{\vrule \vbox to8pt{\hrule width 6pt
\vfill\hrule}\vrule}\par}
\definecolor{darkblue}{rgb}{0.05, .05, .65}
\definecolor{darkgreen}{rgb}{0.1, .65, .1}
\definecolor{darkred}{rgb}{0.8,0,0}
\newcommand{\beqn}{\begin{equation}}
\newcommand{\eeqn}{\end{equation}}
\newcommand{\bear}{\begin{eqnarray}}
\newcommand{\eear}{\end{eqnarray}}
\newcommand{\bean}{\begin{eqnarray*}}
\newcommand{\eean}{\end{eqnarray*}}
\begin{document}
%%%%%%%%%%%%%%%%%%%%
%%%%%%%%%%%%%%%%%%%%

%%%%%%%%%%%%%%%%%%%%
%%%%%%%%%%%%%%%%%%%%
\title{\huge \bf A porous medium equation with spatially inhomogeneous absorption. Part I: Self-similar solutions}
\author{
\Large Razvan Gabriel Iagar\,\footnote{Departamento de Matem\'{a}tica
Aplicada, Ciencia e Ingenieria de los Materiales y Tecnologia
Electr\'onica, Universidad Rey Juan Carlos, M\'{o}stoles,
28933, Madrid, Spain, \textit{e-mail:} razvan.iagar@urjc.es}
\\[4pt] \Large Diana-Rodica Munteanu\,\footnote{Faculty of Psychology and Educational Sciences, Ovidius University of Constanta, 900527, Constanta, Romania, \textit{e-mail:} diana.rodica.merlusca@gmail.com}\\
}
\date{}
\maketitle
%%%%%%%%%%%%%%%%%%%%
%%%%%%%%%%%%%%%%%%%%

%%%%%%%%%%%%%%%%%%%%
%%%%%%%%%%%%%%%%%%%%
\begin{abstract}
This is the first of a two-parts work on the qualitative properties and large time behavior for the following quasilinear equation involving a spatially inhomogeneous absorption
$$
\partial_tu=\Delta u^m-|x|^{\sigma}u^p,
$$
posed for $(x,t)\in\real^N\times(0,\infty)$, $N\geq1$, and in the range of exponents $1<m<p<\infty$, $\sigma>0$. We give a complete classification of (singular) self-similar solutions of the form
$$
u(x,t)=t^{-\alpha}f(|x|t^{-\beta}), \ \alpha=\frac{\sigma+2}{\sigma(m-1)+2(p-1)}, \ \beta=\frac{p-m}{\sigma(m-1)+2(p-1)},
$$
showing that their form and behavior strongly depends on the critical exponent
$$
p_F(\sigma)=m+\frac{\sigma+2}{N}.
$$
For $p\geq p_F(\sigma)$, we prove that all self-similar solutions have a tail as $\xi\to\infty$ of one of the forms
$$
u(x,t)\sim C|x|^{-(\sigma+2)/(p-m)} \quad {\rm or} \quad u(x,t)\sim \left(\frac{1}{p-1}\right)^{1/(p-1)}|x|^{-\sigma/(p-1)},
$$
while for $m<p<p_F(\sigma)$ we add to the previous the \emph{existence and uniqueness} of a \emph{compactly supported very singular solution}. These solutions will be employed in describing the large time behavior of general solutions in a forthcoming paper.
\end{abstract}

\smallskip


\smallskip


\section{Introduction and main results}

The aim of this paper (and the forthcoming companion work Part II) is to establish qualitative properties and a study of the large time behavior for the following quasilinear absorption-diffusion equation
\begin{equation}\label{eq1}
\partial_tu=\Delta u^m-|x|^{\sigma}u^p, \qquad (x,t)\in\real^N\times(0,\infty),
\end{equation}
posed for $1<m<p<\infty$ and $\sigma>0$. This is a generalization, involving a spatially inhomogeneous weight, of an absorption-diffusion equation that is well understood by now, namely
\begin{equation}\label{eq2}
\partial_tu=\Delta u^m-u^p, \qquad (x,t)\in\real^N\times(0,\infty).
\end{equation}
We stress here that, while the current work, devoted to the classification of self-similar solutions, which requires a structure that is invariant to a specific rescaling, only deals with \eqref{eq1}, in the companion work dealing with the large time behavior of solutions, we may also consider more general weights $\varrho(x)$ with suitable properties instead of pure powers $|x|^{\sigma}$.

The most interesting feature of Eq. \eqref{eq1} is the competition for governing the dynamics of it, between the effects of the terms in the right hand side. Indeed, on the one hand, the diffusion of porous medium type expands the support of any compactly supported solution while conserving its $L^1$ norm, while, on the other hand, the spatially inhomogeneous absorption term implies a decrease of the $L^1$ norm of a solution to Eq. \eqref{eq1}, influenced strongly by regions where $|x|$ is large. This competition gives rise to a number of ranges, limited by critical exponents, in which the typical behavior of a solution (usually given by the particular cases of self-similar solutions) varies strongly. This is how different mathematical phenomena, as described below, occur.

Our range of interest in the present paper, that is, $p>m>1$, originated a lot of interesting research in the final part of the past century in the quest to understand the mathematical analysis of the solutions to Eq. \eqref{eq2}. One important critical exponent $p_F(0)=m+2/N$, known as the \emph{Fujita exponent} and identified originally in reaction-diffusion problems starting from the seminal work by Fujita \cite{Fu66}, plays a significant role also for our absorption-diffusion problem. Indeed, it has been proved that, either the porous medium equation governs the dynamics of Eq. \eqref{eq2} as $t\to\infty$ if $p\geq p_F(0)$, or there is a balance between the two terms leading to new asymptotic profiles in the form of very singular self-similar solutions, in the range $m<p<p_F(0)$. In particular, the problem of identifying self-similar solutions in the latter range has been considered in papers such as \cite{BPT86, PT86, KP86, KV88, KPV89, Le97}, where a number of different \emph{very singular solutions} have been deduced, depending on their initial trace as $t\to0$. By \emph{very singular solution}, we understand a solution (in weak or classical sense) to a partial differential equation (in our case \eqref{eq1} or \eqref{eq2}) having the following properties:
\begin{equation}\label{VSS1}
\lim\limits_{t\to0}\sup\limits_{|x|>\epsilon}u(x,t)=0,
\end{equation}
and
\begin{equation}\label{VSS2}
\lim\limits_{t\to0}\int_{|x|<\epsilon}u(x,t)\,dx=+\infty,
\end{equation}
for any $\epsilon>0$. It is by now understood that such solutions appear when a balance between the two terms in competition in the equation is achieved. Moreover, this class of solutions have had a great importance in understanding how general solutions to Eq. \eqref{eq2} behave, as considered in works such as \cite{KP86, KU87, Kwak98, MPV91} and references therein, where the large time behavior of solutions to Eq. \eqref{eq2} as $t\to\infty$ is established. It has been thus shown in the range $m<p<p_F(0)$ that the asymptotic profiles are self-similar solutions, which may differ from compactly supported to a specific tail at infinity depending on the decay of the initial condition $u_0(x)$ as $|x|\to\infty$. Later on, very singular solutions in self-similar form have been established, emphasizing their importance for the large time behavior of integrable solutions, for a number of different equations involving a fast diffusion (that is, \eqref{eq2} with $m<1$) in \cite{PZ91, Le96, Kwak98b, CQW05}, a $p$-Laplacian term in, for example, \cite{PW88, KV92, CQW03, CQW03b, CQW07} and an absorption in the form of a gradient term, see for example \cite{BL01, BKL04, Shi04, IL13, IL14} and references therein.

Motivated by problems in mathematical biology as following from \cite{GMC77, N80}, equations mixing a porous medium diffusion and an absorption involving a spatially inhomogeneous weight have been considered at first by Peletier and Tesei in \cite{PT85, PT86b}, where the authors study, in dimension $N=1$, equations of the form
$$
\partial_tu=(u^m)_{xx}-a(x)u^p,
$$
under some conditions on the weight $a(x)$. The latter mentioned works are devoted to the threshold between positivity (that is, expansion of the support of a compactly supported data covering $\real^N$ as $t\to\infty$) and localization, which depends on whether $p>m$ or $p<m$. Later, Belaud and Shishkov \cite{Belaud01, BeSh07, BeSh22} studied the phenomenon of finite time extinction for absorption-diffusion equations involving more general weights than $|x|^{\sigma}$ but for the so-called range of strong absorption, that is, $0<p<1$. Still considering Eq. \eqref{eq1} with $0<p<1$, the large time behavior of general solutions towards some self-similar solutions, and conditions for finite time extinction, have been established in the recent papers \cite{IL23, ILS24}. Thus, in the present paper we extend the (well established for Eq. \eqref{eq2}) classification of solutions in self-similar form (either very singular or not) to Eq. \eqref{eq1}, their application to large time behavior being left for a forthcoming work.

\medskip

\noindent \textbf{Main results.} Our main object of study will be the radially symmetric self-similar solutions to Eq. \eqref{eq1} in the following form:
\begin{equation}\label{SSS}
u(x,t)=t^{-\alpha}f(\xi), \qquad \xi=|x|t^{-\beta}, \qquad (x,t)\in\real^N\times(0,\infty).
\end{equation}
Inserting the ansatz \eqref{SSS} into Eq. \eqref{eq1} and taking into account that $p>m>1$, we obtain by direct calculations that
\begin{equation}\label{SSexp}
\alpha=\frac{\sigma+2}{L}>0, \qquad \beta=\frac{p-m}{L}>0, \qquad L=\sigma(m-1)+2(p-1)
\end{equation}
are the self-similarity exponents, while the profiles $f(\xi)$ are solutions to the differential equation
\begin{equation}\label{SSODE}
(f^m)''(\xi)+\frac{N-1}{\xi}(f^m)'(\xi)+\alpha f(\xi)+\beta\xi f'(\xi)-\xi^{\sigma}f^p(\xi)=0,
\end{equation}
together with the initial conditions (the second one being imposed by the radial symmetry)
\begin{equation}\label{IC}
f(0)=A>0, \qquad f'(0)=0.
\end{equation}
Letting $F=f^m$ and applying the Cauchy-Lipschitz theorem to the Cauchy problem \eqref{SSODE}-\eqref{IC}, we obtain that, for any $A>0$, there is a unique solution $f(\cdot;A)$ which is positive in a maximal interval (of positivity) $[0,\xi_{\rm max}(A))$ with the property that
$$
F(\cdot;A)=f^m(\cdot;A)\in C^2([0,\xi_{\rm max}(A)),
$$
where either $\xi_{\rm max}(A)=\infty$ or $\xi_{\rm max}(A)<\infty$ and, in the latter case, we say that we have a \emph{compactly supported solution} if moreover $(f^m)'(\xi_{\rm max}(A);A)=0$. With this discussion in mind, we are in a position to state our results concerning the classification of the self-similar solutions $f(\cdot;A)$ according to their behavior as $\xi\to\xi_{\rm max}(A)$. First of all, we can establish a general result valid for any $p>m$.
\begin{theorem}\label{th.gen}
Let $p>m>1$, $N\geq1$. Then, there exists $A^*\in(0,\infty)$ such that

(a) For any $A\in(0,A^*]$, the solution $f(\cdot;A)$ to the Cauchy problem \eqref{SSODE}-\eqref{IC} is decreasing on its positivity region. Moreover, the limiting solution $f(\cdot;A^*)$ satisfies $f(\xi;A^*)>0$ for any $\xi>0$ and has the precise behavior at infinity
\begin{equation}\label{decay.Pg}
\lim\limits_{\xi\to\infty}\xi^{\sigma/(p-1)}f(\xi;A^*)=\left(\frac{1}{p-1}\right)^{1/(p-1)}.
\end{equation}

(b) For any $A\in(A^*,\infty)$, the solution $f(\cdot;A)$ to the Cauchy problem \eqref{SSODE}-\eqref{IC} has a unique positive minimum point $\xi_0(A)$ such that: $f(\xi_0(A);A)>0$, $f(\cdot;A)$ is decreasing for $\xi\in(0,\xi_0(A))$ and increasing for $\xi\in(\xi_0(A),\infty)$.
\end{theorem}
Such an exhaustive classification has been given for the non-weighted case $\sigma=0$ by Leoni in \cite{Le97}. However, when $\sigma>0$ things are more involved, since for $\sigma=0$, there exists a constant solution to \eqref{SSODE}-\eqref{IC}, namely
\begin{equation}\label{const.sol}
f(\xi)=\left(\frac{1}{p-1}\right)^{1/(p-1)},
\end{equation}
which both corresponds to $f(\cdot;A^*)$ and to the local behavior \eqref{decay.Pg} (which for $\sigma=0$ is no longer a decay, but a constant behavior). In our case, such an explicit limiting solution ceases to exist for Eq. \eqref{eq1} and we have thus to prove the existence of a non-explicit one by different techniques. Moreover, while for $\sigma=0$, profiles $f(\cdot;A)$ with $A>A^*$ are increasing always (as pointed out in \cite{Le97}), in our case they start in a decreasing way in a right neighborhood of the origin, reach a positive minimum and then become increasing forever.

We are now left with the classification of the profiles $f(\cdot;A)$ with $A\in(0,A^*)$, according to their behavior as $\xi\to\xi_{\rm max}(A)$ (either finite or infinite). As already noticed in the case of Eq. \eqref{eq2}, this analysis will strongly depend on the \emph{Fujita-type exponent}
\begin{equation}\label{pFs}
p_F(\sigma)=m+\frac{\sigma+2}{N},
\end{equation}
which has been first analyzed in connection with reaction-diffusion problems and the phenomenon of finite time blow-up of solutions in, for example, \cite{Qi98, Su02}. The richest case is when $m<p<p_F(\sigma)$, when the classification is given in the next result.
\begin{theorem}\label{th.small}
Let $m>1$, $\sigma>0$ and $p\in(m,p_F(\sigma))$, where $p_F(\sigma)$ is defined in \eqref{pFs}.

\medskip

(a) There exists a unique decreasing, compactly supported self-similar profile solving the problem \eqref{SSODE}-\eqref{IC}, that is, there exists a unique $A_*\in(0,A^*)$ such that $\xi_{\rm max}(A_*)<\infty$, $f(\cdot;A_*)$ is decreasing on $[0,\xi_{\rm max}(A_*)]$ and $(f^m)'(\xi_{\rm max}(A_*);A_*)=0$.

\medskip

(b) There exists $A_0\in(A_*,A^*]$ such that for any $A\in(A_*,A_0)$ we have $f(\xi;A)>0$ for any $\xi\in[0,\infty)$ and it has the decay
\begin{equation}\label{decay.P1}
f(\xi;A)\sim C\xi^{-(\sigma+2)/(p-m)}, \qquad {\rm as} \ \xi\to\infty,
\end{equation}
for some constant $C>0$ (that might depend on $A$). For any $A\in[A_0,A^*]$, the profile $f(\cdot;A)$ presents the decay \eqref{decay.Pg} as $\xi\to\infty$.
\end{theorem}
Let us stress here that the solution in the form \eqref{SSS} with self-similar profile as in Theorem \ref{th.small}, part (a), is a \emph{very singular solution}. Indeed, for any $x\in\real^N$, $|x|>0$, we have $|x|t^{-\beta}\to\infty$ as $t\to0$, and thus condition \eqref{VSS1} is fulfilled, owing to the compact support of the profile $f$, while condition \eqref{VSS2} follows from a direct calculation, based on the fact that
$$
N\beta-\alpha=\frac{N(p-m)-(\sigma+2)}{L}<0.
$$
In the meantime, the self-similar solutions \eqref{SSS} with profiles decaying as in \eqref{decay.P1} have as initial trace
\begin{equation}\label{init.trace}
\lim\limits_{t\to0}u(x,t)=C|x|^{-(\sigma+2)/(p-m)}, \qquad {\rm for} \ |x|>0,
\end{equation}
while condition \eqref{VSS2} is still satisfied. Despite the fact that these solutions have a tail instead of a compact support, their importance for the large time behavior of general solutions for the homogeneous case $\sigma=0$ given as Eq. \eqref{eq2} has been emphasized, for example, in \cite{Kwak98} (see also \cite{KP86}). Let us also mention here, in order to complete the panorama, that for $A\in(0,A_*)$, the profile $f(\cdot;A)$ does not give rise to a non-negative self-similar solution, as it changes sign at $\xi_{{\rm max}}(A)\in(0,\infty)$ in the sense that $f(\xi_{{\rm max}}(A);A)=0$ but $(f^m)'(\xi_{{\rm max}}(A);A)<0$.

Things are much simpler in the range $p\geq p_F(\sigma)$.
\begin{theorem}\label{th.large}
Let $m>1$, $\sigma>0$ and $p\geq p_F(\sigma)$. Then, for any $A\in(0,A^*)$ we have $f(\xi;A)>0$ for any $\xi\in(0,\infty)$. Moreover, there exists $A_0\in(0,A^*]$ such that for any $A\in(0,A_0)$, $f(\cdot;A)$ has the decay \eqref{decay.P1} as $\xi\to\infty$, while for any $A\in[A_0,A^*]$, the profile $f(\cdot;A)$ presents the decay \eqref{decay.Pg} as $\xi\to\infty$.
\end{theorem}
Observe that the solutions $u$ in the form \eqref{SSS} with profiles $f(\cdot;A)$, presenting once more a tail as $\xi\to\infty$ and an initial trace as in \eqref{init.trace} for $|x|>0$, do no longer satisfy the condition \eqref{VSS2}, since now $N\beta-\alpha\geq0$. Moreover, since
$$
-\frac{\sigma+2}{p-m}=-\frac{\alpha}{\beta}\geq-N,
$$
we also infer that $f(\cdot;A)\notin L^1([0,\infty))$ and thus the corresponding solutions to profiles as in Theorem \ref{th.large} are no longer integrable. But they are still classical solutions to Eq. \eqref{eq1}.

\medskip

\noindent \textbf{Conjecture}. We strongly expect that $A_0=A^*$, that is, the decreasing profile $f(\cdot;A^*)$ with decay \eqref{decay.Pg} to be unique. However, this seems to be difficult to prove rigorously, and in particular a technique based on some analysis of a linear operator employed with success in \cite{CQW03b} and in previous works by one of the authors such as \cite{IL13} apparently fails here because of a lack of homogeneity of the corresponding operator precisely caused by the presence of the weight $\xi^{\sigma}$. We thus leave this uniqueness question as an \emph{open problem}.

\medskip

\noindent \textbf{Remark.} The above Theorems remain partially true if we allow $\sigma\in(-2,0)$. Indeed, the existence and uniqueness of solutions as stated in Theorems \ref{th.small} and \ref{th.large} still hold true, but the $C^2$ property in a neighborhood of $\xi=0$ (and thus, the property of being classical solutions) is lost in this range. Moreover, in the range $-2<\sigma\leq-1$ we even lose the initial condition $f'(0)=0$ in \eqref{IC}, and thus, the profiles do no longer give rise to solutions in the standard weak sense at $\xi=0$. We refrain from entering this range in the present work.

\medskip

\noindent \textbf{Organization of the paper.} Instead of a standard shooting method, whose adaptation from \cite{Le97} might be more tedious due to the presence of the variable coefficient $|x|^{\sigma}$ and the extra difficulties it involves (as commented after the statement of Theorem \ref{th.gen}), the proofs of our main results rely on a shooting method on a transformed version of the equation \eqref{SSODE} into a three-dimensional autonomous dynamical system, transformation that has been employed with success by one of the authors in the study of reaction-diffusion equations in recent works such as \cite{IS22, IMS23, ILS24b} and which has the advantage of giving also a ``visual" understanding of how the limiting cases $A=A^*$, respectively $A=A_*$, come into play. The local analysis of the dynamical system is performed in Section \ref{sec.transf}, followed by some preparatory results on the global analysis in Section \ref{sec.prep}. Uniqueness follows from a monotonicity of the decreasing profiles $f(\cdot;A)$ with respect to the initial condition $f(0)$, which is established in Section \ref{sec.monot}. Finally, after all these preparations, the proofs of our main results are given in Section \ref{sec.thm}, closing the paper.

\section{The transformation. An autonomous dynamical system}\label{sec.transf}

We consider the following transformation, which has been employed with success in previous works on reaction-diffusion equations (see for example \cite{IMS23}):
\begin{equation}\label{PSchange}
X(\xi)=\frac{m}{\alpha}\xi^{-2}f(\xi)^{m-1}, \ \ Y(\xi)=\frac{m}{\alpha}\xi^{-1}f(\xi)^{m-2}f'(\xi), \ \ Z(\xi)=\frac{1}{\alpha}\xi^{\sigma}f(\xi)^{p-1},
\end{equation}
where the new independent variable $\eta$ is introduced in an implicit way via the differential equation
\begin{equation}\label{ind.var}
\frac{d\eta}{d\xi}=\frac{\alpha}{m}\xi f(\xi)^{1-m}=\frac{1}{\xi X(\xi)}.
\end{equation}
Noticing that the second formula in \eqref{PSchange} gives
$$
f'(\xi)=\frac{\alpha}{m}\xi Y(\eta)f^{2-m}(\xi), \qquad (f^m)'(\xi)=\alpha Y(\eta)\xi f(\xi)
$$
and
$$
(f^m)''(\xi)=\alpha\left(\xi f(\xi)\frac{dY}{d\xi}+\frac{\alpha}{m}\xi f^{2-m}(\xi)Y^2(\eta)+Y(\eta)f(\xi)\right),
$$
we replace the above formulas into \eqref{SSODE} and, after performing some straightforward calculations and pass to derivatives with respect to $\eta$ employing \eqref{ind.var}, we are left with the following autonomous three-dimensional dynamical system
\begin{equation}\label{PSsyst1}
\left\{\begin{array}{ll}\dot{X}=X[(m-1)Y-2X],\\
\dot{Y}=-Y^2-\frac{p-m}{\sigma+2}Y-X-NXY+XZ,\\
\dot{Z}=Z[(p-1)Y+\sigma X],\end{array}\right.
\end{equation}
where the dot derivatives are taken with respect to $\eta$. Related to it, and in order to visualize and study the limit $X\to\infty$ in the previous dynamical system, we also consider a further change of variable
\begin{equation}\label{PSchange.inf}
x=\frac{1}{X}, \quad y=\frac{Y}{X}, \quad z=\frac{Z}{X}, \quad \frac{d\eta_1}{d\eta}=X(\eta),
\end{equation}
which in terms of profiles writes
\begin{equation}\label{PSchange2}
x(\xi)=\frac{\alpha}{m}\xi^2f(\xi)^{1-m}, \qquad y(\xi)=\frac{\xi f'(\xi)}{f(\xi)}, \qquad z(\xi)=\frac{1}{m}\xi^{\sigma+2}f^{p-m}(\xi),
\end{equation}
together with the independent variable $\eta_1=\ln\,\xi$. Replacing the change of variable \eqref{PSchange.inf} in the system \eqref{PSsyst1} and taking derivatives with respect to $\eta_1$, we easily obtain that $(x,y,z)$ satisfy the autonomous system
\begin{equation}\label{PSsyst2}
\left\{\begin{array}{ll}\dot{x}=x(2-(m-1)y), \\ \dot{y}=-x-(N-2)y+z-my^2-\frac{p-m}{\sigma+2}xy, \\ \dot{z}=z(\sigma+2+(p-m)y).\end{array}\right.
\end{equation}
Notice that the systems \eqref{PSsyst1} and \eqref{PSsyst2} are dual one to the other, in the sense that one more application of the change of variable \eqref{PSchange.inf} to \eqref{PSsyst2} gets back to \eqref{PSsyst1}. Moreover, according to the theory of the Poincar\'e sphere (see for example \cite[Theorem 5 (a), Section 3.10]{Pe}), the second system represents the limit as $X\to\infty$ of the first system and it will be used to analyze locally the critical points at infinity of it. Let us observe that, since we are only interested in non-negative solutions, we have $X\geq0$, $Z\geq0$ in \eqref{PSsyst1} (respectively $x\geq0$, $z\geq0$ in \eqref{PSsyst2}) and the planes $\{X=0\}$ and $\{Z=0\}$ (respectively $\{x=0\}$ and $\{z=0\}$) are invariant for \eqref{PSsyst1} (respectively \eqref{PSsyst2}).

\subsection{Critical points of the system \eqref{PSsyst1}}

Equating the right-hand side of \eqref{PSsyst1} to zero, we obtain the following critical points, all them lying in the plane $\{X=0\}$:
\begin{equation}\label{crit.syst1}
P_1=(0,0,0), \qquad P_2=\left(0,-\frac{p-m}{\sigma+2},0\right), \qquad P_{\gamma}=(0,0,\gamma), \qquad \gamma\in(0,\infty).
\end{equation}
We analyze below the local behavior of the trajectories of the system \eqref{PSsyst1} near these points.
\begin{lemma}\label{lem.P1}
The critical point $P_1$ is a non-hyperbolic point having a one-dimensional stable manifold and two-dimensional center manifolds with stable direction of the flow, forming thus a three dimensional center-stable manifold. The trajectories entering $P_1$ on the center-stable manifold correspond to profiles having the local behavior \eqref{decay.P1}.
\end{lemma}
\begin{proof}
The linearization of the system \eqref{PSsyst1} near $P_1$ has the matrix
$$
M(P_1)=\left(
         \begin{array}{ccc}
           0 & 0 & 0 \\
           -1 & -\frac{p-m}{\sigma+2} & 0 \\
           0 & 0 & 0 \\
         \end{array}
       \right),
$$
leading to a one-dimensional stable manifold and two dimensional center manifolds (which may not be unique). In order to study the center manifold, we replace $Y$ by the new variable
\begin{equation}\label{cmf}
W:=X+\frac{p-m}{\sigma+2}Y, \qquad {\rm or \ equivalently}, \qquad Y=\frac{\sigma+2}{p-m}(W-X),
\end{equation}
obtaining, after some direct calculations, the new system
\begin{equation}\label{interm1}
\left\{\begin{array}{ll}\dot{X}=-\frac{1}{\beta}X^2+\frac{(m-1)\alpha}{\beta}XW,\\
\dot{W}=-\frac{\beta}{\alpha}W-\frac{\alpha}{\beta}W^2-\frac{-\alpha(m+1)+N\beta}{\beta}XW+\frac{\beta}{\alpha}XZ+\frac{(N-2)\beta-m\alpha}{\beta}X^2,\\
\dot{Z}=-\frac{1}{\beta}XZ+\frac{\alpha(p-1)}{\beta}ZW.\end{array}\right.
\end{equation}
Following \cite[Section 2.5]{Carr}, we look for a second order Taylor approximation of the center manifolds in the form
$$
W=aX^2+bXZ+cZ^2+o(|(X,Z)|^3),
$$
with coefficients $a$, $b$, $c$ to be determined. A direct substitution of the previous ansatz in the equation of the center manifold (see for example \cite[Theorem 1, Section 2.12]{Pe}) and employing the system \eqref{interm1} leads, by equating terms of the same degree in the resulting equation, to the following coefficients
\begin{equation*}
a=\frac{\sigma+2}{(p-m)^2}\left[p(N-2)-m(N+\sigma)\right], \qquad b=1, \qquad c=0,
\end{equation*}
while a simple induction then shows that, furthermore, there will be no terms of pure powers of $Z$ (as there are none in the second equation of \eqref{interm1}), that is, the center manifold has the local approximation
\begin{equation}\label{centerP1}
W=\frac{\sigma+2}{(p-m)^2}\left[p(N-2)-m(N+\sigma)\right]X^2+XZ+XO(|(X,Z)|^2).
\end{equation}
Moreover, according to the Reduction Principle (see \cite[Section 2.4]{Carr}), the direction of the flow on the center manifolds is given by the reduced system maintaining only the second degree, dominating terms in the first and third equation of the system \eqref{interm1} after replacing $W$ by its approximation \eqref{centerP1}, namely
\begin{equation}\label{interm2}
\left\{\begin{array}{ll}\dot{X}&=-\frac{1}{\beta}X^2+X^2O(|(X,Z)|),\\
\dot{Z}&=-\frac{1}{\beta}XZ+XO(|(X,Z)|^2),\end{array}\right.
\end{equation}
in a neighborhood of its origin $(X,Z)=(0,0)$. We thus infer that the flow goes into the stable direction on every center manifold, and that there are infinitely many center manifolds (according to, for example, the theory in \cite[Section 3]{Sij}), forming together with the stable manifold, a center-stable manifold of dimension three. The trajectories contained in this center-stable manifold have a local behavior obtained, in a first approximation, by the integration of the system \eqref{interm2}, which leads to
$$
Z(\eta)\sim KX(\eta), \qquad {\rm as} \ \eta\to\infty, \qquad K>0,
$$
which in terms of profiles leads to \eqref{decay.P1}, by undoing \eqref{PSchange}.
\end{proof}
We turn now our attention to the critical point $P_2$.
\begin{lemma}\label{lem.P2}
The critical point $P_2$ is a (hyperbolic) saddle point, with a two-dimensional stable manifold and a one-dimensional unstable manifold contained in the $Y$ axis. The trajectories contained in the two-dimensional stable manifold correspond to profiles presenting an interface at some point $\xi_0\in(0,\infty)$ with the precise local behavior
\begin{equation}\label{beh.P2}
f(\xi)\sim\left[C-\frac{\beta(m-1)}{2m}\xi^2\right]_{+}^{1/(m-1)}, \qquad {\rm as} \ \xi\to\xi_0=\sqrt{\frac{2mC}{\beta(m-1)}}, \ \xi<\xi_0,
\end{equation}
where $C>0$ is a free constant.
\end{lemma}
\begin{proof}
The linearization of \eqref{PSsyst1} near $P_2$ has the matrix
$$
M(P_2)=\left(
         \begin{array}{ccc}
           -\frac{(m-1)\beta}{\alpha} & 0 & 0 \\
           -1+\frac{N\beta}{\alpha} & \frac{\beta}{\alpha} & 0 \\
           0 & 0 & -\frac{(p-1)\beta}{\alpha} \\
         \end{array}
       \right)
$$
with eigenvalues $\lambda_1=-(m-1)\beta/\alpha$, $\lambda_2=\beta/\alpha$ and $\lambda_3=-(p-1)\beta/\alpha$, and corresponding eigenvectors
$$
e_1=\left(1,\frac{N\beta-\alpha}{m\beta},0\right), \ e_2=(0,1,0), \ e_3=(0,0,1).
$$
The invariance of the $Y$-axis (as intersection of two invariant plane) together with the uniqueness of the unstable manifold given by \cite[Theorem 3.2.1]{GH}, prove that the unstable manifold is contained in the $Y$ axis. The stable manifold is two-dimensional and tangent to the plane spanned by the eigenvectors $e_1$ and $e_3$. The trajectories entering $P_1$ on the stable manifold correspond to profiles such that $Y(\eta)\to-\beta/\alpha$ as $\eta\to\infty$, which leads by undoing \eqref{PSchange} and integration to \eqref{beh.P2}. We readily notice that \eqref{beh.P2} implies the interface condition $f(\xi_0)=0$, $(f^m)'(\xi_0)=0$ and $f(\xi)>0$ for $\xi$ in a left-neighborhood of $\xi_0$.
\end{proof}
Finally, the analysis of the critical points $P_{\gamma}$ with $\gamma>0$ leads to the appearance of the local behavior \eqref{decay.Pg} as $\xi\to\infty$.
\begin{lemma}\label{lem.Pg}
Letting
$$
\gamma_0:=\frac{1}{\alpha(p-1)},
$$
the critical point $P_{\gamma_0}$ has a two-dimensional center-stable manifold with trajectories arriving from the region $\{X>0, Z>0\}$ of the phase space. These trajectories correspond to profiles with local behavior given by \eqref{decay.Pg} as $\xi\to\infty$. For any $\gamma>0$, $\gamma\neq\gamma_0$, there are no trajectories of the system \eqref{PSsyst1} entering $P_{\gamma}$ from the region $\{X>0\}$ of the phase space.
\end{lemma}
\begin{proof}
We give here first a direct, but formal argument in terms of profiles. Assume that there is $\gamma\in(0,\infty)$ and a trajectory entering the point $P_{\gamma}$ from the region $\{X>0\}$ of the phase space. Since $Z(\eta)\to\gamma$ as $\eta\to\infty$, it is easy to see (by undoing the first and third definitions in \eqref{PSchange}) that this trajectory is locally mapped into a profile $f(\xi)$ with local behavior
$$
f(\xi)\sim K\xi^{-\sigma/(p-1)}, \qquad {\rm as} \ \xi\to\infty, \qquad K=(\alpha\gamma)^{1/(p-1)}>0.
$$
Assuming (at a formal level) that also the derivative behaves as
$$
f'(\xi)\sim-\frac{K\sigma}{p-1}\xi^{-\sigma/(p-1)-1}, \qquad {\rm as} \ \xi\to\infty,
$$
and introducing these first order approximations into the differential equation \eqref{SSODE}, we infer that, in a first approximation, the dominating order $\xi^{-\sigma/(p-1)}$ is given by the three last terms in \eqref{SSODE} and we get, as $\xi\to\infty$,
\begin{equation*}
\begin{split}
\alpha f(\xi)+\beta\xi f'(\xi)-\xi^{\sigma}f^p(\xi)&\sim K\left(\alpha-\frac{\beta\sigma}{p-1}-K^{p-1}\right)\xi^{-\sigma/(p-1)}\\
&\sim K\left(\frac{1}{p-1}-K^{p-1}\right)\xi^{-\sigma/(p-1)},
\end{split}
\end{equation*}
hence the only possibility to cancel out the first order approximation is to take $K^{p-1}=1/(p-1)$, which leads to $\gamma=\gamma_0$, and to the local behavior \eqref{decay.Pg}, as claimed. A rigorous proof is based on the analysis of the center manifold of the critical point $P_{\gamma}$ for any $\gamma\in(0,\infty)$, which is rather technical, based on the following change of variable in the system \eqref{PSsyst1}:
$$
(X,Y,Z)\mapsto (X,W,V), \qquad W=\frac{\beta}{\alpha}Y+(1-\gamma)X, \qquad V=Z-\gamma-kY, \qquad k=\frac{(p-1)\alpha\gamma}{\beta},
$$
and the calculations leading to the analysis of the center manifold follow very closely the lines of the proof of \cite[Lemma 3.2]{IS24} (see also the ones in \cite[Lemma 2.4]{ILS24b}), the only difference with respect to these references being the fact that the only nonzero eigenvalue of the linearization of \eqref{PSsyst1} near $P_{\gamma}$ is now negative, which leads to the formation of a center-stable two-dimensional manifold. We omit here the very similar details and we refer the reader to the quoted references.
\end{proof}

%\noindent \textbf{Remark.} For $\sigma=0$, the profile in Lemma \ref{lem.Pg} is r constant.

\subsection{Critical points of the system \eqref{PSsyst2}}

Equating the right-hand side of \eqref{PSsyst2} to zero, we obtain the following critical points, all them lying in the plane $\{x=0\}$:
\begin{equation}\label{crit.syst2}
Q_1=(0,0,0), \qquad Q_2=\left(0,-\frac{N-2}{m},0\right), \qquad Q_3=\left(0,-\frac{\sigma+2}{p-m},Z_0\right),
\end{equation}
with
\begin{equation}\label{Z0}
Z_0=\frac{(\sigma+2)[m(N+\sigma)-p(N-2)]}{(p-m)^2},
\end{equation}
the latter of them existing only for $m<p<m(N+\sigma)/(N-2)$. We analyze below the local behavior of the trajectories of the system \eqref{PSsyst2} near these points. To fix the ideas, let us work for now in dimension $N\geq3$. The critical point $Q_1$ is the most interesting for our study.
\begin{lemma}\label{lem.Q1}
The critical point $Q_1$ is a saddle point in the system \eqref{PSsyst2} with a one-dimensional stable manifold contained in the $y$ axis and a two-dimensional unstable manifold. The trajectories on the unstable manifold form a one-parameter family with first approximation
\begin{equation}\label{lC}
(l_C): \ y(\eta_1)\sim-\frac{x(\eta_1)}{N}, \qquad z(\eta_1)\sim Cx(\eta_1)^{(\sigma+2)/2}, \qquad C\in[0,\infty)
\end{equation}
as $\eta_1\to-\infty$, and correspond to profiles with the local behavior
\begin{equation}\label{beh.Q1}
f(\xi)\sim\left(D-\frac{\alpha(m-1)}{2mN}\xi^2\right)^{1/(m-1)}, \qquad {\rm as} \ \xi\to0, \qquad D\in(0,\infty).
\end{equation}
In particular, the profile $f(\cdot;A)$ solution to \eqref{SSODE}-\eqref{IC} corresponds to the trajectory $l_C$ in the family \eqref{lC} with
\begin{equation}\label{bij}
A=(Cm)^{2/L}\left(\frac{\alpha}{m}\right)^{(\sigma+2)/L}, \qquad L=\sigma(m-1)+2(p-1).
\end{equation}
\end{lemma}
\begin{proof}
The linearization of the system \eqref{PSsyst2} near $Q_1$ has the matrix
$$
\left(
  \begin{array}{ccc}
    2 & 0 & 0 \\
    -1 & -(N-2) & 1 \\
    0 & 0 & \sigma+2 \\
  \end{array}
\right),
$$
with eigenvalues $\lambda_1=2$, $\lambda_2=-(N-2)<0$, $\lambda_3=\sigma+2>0$ and corresponding eigenvectors $e_1=(N,-1,0)$, $e_2=(0,1,0)$ and $e_3=(0,1,N+\sigma)$. The invariance of the $y$ axis in the system \eqref{PSsyst2}, together with the uniqueness of the stable manifold \cite[Theorem 3.2.1]{GH}, prove that the stable manifold is fully contained in the $y$ axis. With respect to the unstable manifold, we deduce from the Stable Manifold Theorem \cite[Section 2.7]{Pe} that the unstable manifold of $Q_1$ is tangent to the vector subspace spanned by the eigenvectors $e_1$ and $e_3$, which readily leads to its linear approximation
\begin{equation}\label{intermXX}
y(\eta_1)=-\frac{x(\eta_1)}{N}+\frac{z(\eta_1)}{N+\sigma}+o(|(x(\eta_1),z(\eta_1))|), \quad {\rm as} \ \eta_1\to-\infty.
\end{equation}
Moreover, we infer from the first and third equation of \eqref{PSsyst2} that, in a first approximation, we have
\begin{equation}\label{interm3}
z(\eta_1)\sim Cx(\eta_1)^{(\sigma+2)/2}, \qquad {\rm as} \ \eta_1\to-\infty,
\end{equation}
for any $C\in[0,\infty)$. This, together with the positivity of $\sigma$, imply that $z(\eta_1)$ is of lower order than $x(\eta_1)$ in a neighborhood of $Q_1$, and we immediately get the approximation \eqref{lC} by neglecting the $z$ term in \eqref{intermXX}. Passing to profiles by undoing \eqref{PSchange2} and recalling that $\eta_1=\ln\,\xi$, we get from \eqref{interm3} and an immediate substitution that the orbits $l_C$ correspond to profiles with $f(0)=A$, with $A$ given by \eqref{bij}. Moreover, the second equation in \eqref{lC} together with \eqref{PSchange2} lead, in a right neighborhood of $\xi=0$, to
$$
(f^{m-1})'(\xi)\sim-\frac{\alpha(m-1)}{mN}\xi,
$$
which, together with $f(0)=A$, lead to the local expansion \eqref{beh.Q1} as $\xi\to0$ after an integration on $(0,\xi)$.
\end{proof}

\noindent \textbf{Notation.} We denote in the sequel by $l_{\infty}$ the unique trajectory belonging to the unstable manifold of $Q_1$ and contained in the plane $\{x=0\}$. Indeed, this is coherent with \eqref{lC}, since if we write
$$
x(\eta_1)\sim \left(\frac{1}{C}z(\eta_1)\right)^{2/(\sigma+2)}, \qquad {\rm as} \ \eta_1\to-\infty,
$$
we notice that $x\equiv0$ corresponds to taking $1/C=0$, that is, $C=\infty$.

\medskip

The critical points $Q_2$ and $Q_3$ are not very interesting for our study, as the following result shows (recalling that for the moment we work in dimension $N\geq3$).
\begin{lemma}\label{lem.Q2Q3}
The critical point $Q_2$ is an unstable node if $m<p<m(N+\sigma)/(N-2)$ and a saddle point with a two-dimensional unstable manifold fully contained in the invariant plane $\{z=0\}$ and a one-dimensional stable manifold fully contained in the invariant plane $\{x=0\}$, if $p>m(N+\sigma)/(N-2)$. The critical point $Q_3$ is a saddle point with a two-dimensional unstable manifold and a one-dimensional stable manifold contained in the plane $\{x=0\}$. The trajectories stemming from these two critical points correspond to profiles presenting a vertical asymptote at $\xi=0$, of the form
\begin{equation}\label{beh.Q23}
f(\xi)\sim\left\{\begin{array}{ll}C\xi^{-(N-2)/m}, & {\rm for} \ Q_2,\\ C\xi^{-(\sigma+2)/(p-m)}, & {\rm for} \ Q_3, \end{array}\right. \quad C>0.
\end{equation}
\end{lemma}
\begin{proof}
The linearization of the system \eqref{PSsyst2} near the critical points $Q_2$, respectively $Q_3$ has the matrix
$$
M(Q_2)=\left(
         \begin{array}{ccc}
           \frac{mN-N+2}{m} & 0 & 0 \\[1mm]
           \frac{(p-m)(N-2)}{m(\sigma+2)}-1 & N-2 & 1 \\[1mm]
           0 & 0 & \frac{m(N+\sigma)-p(N-2)}{m} \\
         \end{array}
       \right)
$$
and
$$
M(Q_3)=\left(
         \begin{array}{ccc}
           \frac{L}{p-m} & 0 & 0 \\[1mm]
           0 & \frac{m(N+2\sigma+2)-p(N-2)}{p-m} & 1 \\[1mm]
           0 & (p-m)Z_0 & 0 \\
         \end{array}
       \right).
$$
It is thus obvious that $M(Q_2)$ has two positive eigenvalues, while the last one changes sign at $p=m(N+\sigma)/(N-2)$, provided $N\geq3$: if $p<m(N+\sigma)/(N-2)$ all three eigenvalues are positive and we have an unstable node, while if $p>m(N+\sigma)/(N-2)$ we have two positive eigenvalues and a negative eigenvalue. A closer inspection of the eigenvectors in this case, together with the invariance of the planes $\{x=0\}$, respectively $\{z=0\}$, lead to the conclusion. With respect to $Q_3$, the first eigenvalue is strictly positive, while the second and third satisfy
$$
\lambda_2\lambda_3=-(p-m)Z_0<0,
$$
so that one is positive and the other is negative. Finally, the local behavior of the profiles corresponding to the orbits going out of these points follows from the fact that, in the case of $Q_3$, we have $Z(\eta_1)\to Z_0$ as $\eta_1\to-\infty$, while in the case of $Q_2$, $Y(\eta_1)\to-(N-2)/m$ as $\eta_1\to-\infty$. Recalling that $\eta_1=\ln\,\xi$, we arrive to \eqref{beh.Q23} by undoing the transformation \eqref{PSchange2} (and an integration on $(0,\xi)$ for $\xi>0$ small, in the case of $Q_2$).
\end{proof}

\noindent \textbf{Remark}. The line $\{y=-(\sigma+2)/(p-m), z=Z_0\}$ is a trajectory of the system \eqref{PSsyst2}, provided $p<m(N+\sigma)/(N-2)$. It corresponds to the explicit singular profile
\begin{equation}\label{stat.sol}
f(\xi)=K\xi^{-(\sigma+2)/(p-m)}, \qquad K=(mZ_0)^{1/(p-m)}.
\end{equation}

\medskip

\noindent \textbf{Dimensions $N=1$ and $N=2$}. This is the only place where letting $N=1$ and $N=2$ introduces a technical change. Indeed, in dimension $N=2$ the critical points $Q_1$ and $Q_2$ coincide, and the resulting point is a saddle-node. However, this does not affect our trajectories $l_C$, as the unstable manifold composed by them and spanned by the eigenvectors $e_1$ and $e_3$ corresponding to eigenvalues $\lambda_1=2$ and $\lambda_3=\sigma+2$ in Lemma \ref{lem.Q1} remains unchanged. In dimension $N=1$, the point $Q_2$ passes to the positive half-space with $y=1/m$, while the critical point $Q_1$ becomes an unstable node. However, we once more distinguish our specific shooting manifold $(l_C)_{C\in(0,\infty)}$ as in the following statement:
\begin{lemma}\label{lem.N1}
The critical point $Q_1$ is an unstable node in dimension $N=1$. The trajectories stemming from $Q_1$ have either the local behavior
\begin{equation}\label{beh.P0.N1}
f(\xi)\sim\left[A-K\xi\right]^{2/(m-1)}, \qquad {\rm as} \ \xi\to0,
\end{equation}
with $A>0$ and $K\in\real\setminus\{0\}$ arbitrary constants, or the local behavior \eqref{beh.Q1}.
\end{lemma}
A complete proof of this fact is completely similar to the analogous one in \cite[Section 6]{IMS23}, to which we refer. However, one can observe before going to the proof that one can still shoot on the two-dimensional manifold spanned by the eigenvectors $e_1$ and $e_3$, which gives the desired behavior \eqref{beh.Q1}.

\subsection{Other critical points at infinity}\label{subsec.inf}

In order for the local analysis of the trajectories of the system \eqref{PSsyst1} (or its ``dual" \eqref{PSsyst2}) to be complete, we have to perform an analysis of the critical points at infinity. Following, for example, the theory in \cite[Section 3.10]{Pe}, this is done by passing to the Poincar\'e hypersphere in four variables by setting
$$
X=\frac{\overline{X}}{W}, \qquad Y=\frac{\overline{Y}}{W}, \qquad Z=\frac{\overline{Z}}{W}.
$$
The critical points at infinity of the system \eqref{PSsyst1}, expressed in these new variables, are then given by the following system (according, for example, to \cite[Theorem 4, Section 3.10]{Pe}):
\begin{equation}\label{Poincare}
\left\{\begin{array}{ll}\overline{X}[\overline{X}\overline{Z}-(N-2)\overline{X}\overline{Y}-m\overline{Y}^2]=0,\\
\overline{X}\overline{Z}[(\sigma+2)\overline{X}+(p-m)\overline{Y}]=0,\\
\overline{Z}[p\overline{Y}^2+(\sigma+N)\overline{X}\overline{Y}-\overline{X}\overline{Z}]=0,\end{array}\right.
\end{equation}
together with the condition of belonging to the equator of the hypersphere, which implies $W=0$ and thus the additional equation $\overline{X}^2+\overline{Y}^2+\overline{Z}^2=1$. Following \cite[Theorem 5(a), Section 3.10]{Pe}, we find that all the critical points at infinity of \eqref{PSsyst1} with $\overline{X}\neq0$ correspond to critical points of the system \eqref{PSsyst2}, and we thus find the points $Q_1$, $Q_2$ and $Q_3$ already analyzed in the previous section. Apart from these, we can let $\overline{X}=0$ in \eqref{Poincare} and find that either $\overline{Z}=0$ or $\overline{Y}=0$. We thus find three more critical points at infinity, namely
\begin{equation*}
Q_4=(0,0,1,0), \qquad Q_5=(0,-1,0,0), \qquad Q_6=(0,1,0,0).
\end{equation*}
Let us recall here that, in terms of the variables $(X,Y,Z)$, the critical point $Q_5$ is characterized by trajectories such that
\begin{equation}\label{loc.Q5}
Y(\eta)\to-\infty, \qquad \frac{X(\eta)}{Y(\eta)}\to0, \qquad \frac{Z(\eta)}{Y(\eta)}\to0, \qquad {\rm as} \ \eta\to\infty,
\end{equation}
and a similar characterization holds true for $Q_6$ (but we will not use it in the sequel). We first analyze the flow of the system in the neighborhood of the pair $Q_5$ and $Q_6$. To this end, we follow \cite[Theorem 5(b), Section 3.10]{Pe} to conclude that the flow of \eqref{PSsyst1} near these points is topologically equivalent with the flow near the origin in the following system
\begin{equation}\label{systinf2}
\left\{\begin{array}{ll}\pm\dot{x}=-mx-(N-2)x^2-\frac{\beta}{\alpha}xw-x^2w-x^2z,\\
\pm\dot{z}=-pz-\frac{\beta}{\alpha}zw-(N+\sigma)xz+xz^2-xzw,\\
\pm\dot{w}=-w-\frac{\beta}{\alpha}w^2-xw^2-Nxw+xzw,\end{array}\right.
\end{equation}
where the signs have to be chosen according to the direction of the flow, that is, a plus sign in the system \eqref{systinf2} corresponds to $Q_5$ and a minus sign corresponds to $Q_6$.
\begin{lemma}\label{lem.Q56}
The critical point $Q_5$ is a stable node and the critical point $Q_6$ is an unstable node. The trajectories entering the stable node $Q_5$ correspond to profiles having a compact support such that there is $\xi_0\in(0,\infty)$ and $\delta\in(0,\xi_0)$ with
\begin{equation}\label{beh.Q5}
f(\xi_0)=0, \qquad f(\xi)>0 \ {\rm for} \ \xi\in(\xi_0-\delta,\xi_0), \qquad (f^m)'(\xi_0)<0.
\end{equation}
The trajectories stemming from the unstable node $Q_6$ correspond to profiles such that there is $\xi_0\in(0,\infty)$ and $\delta>0$ with
\begin{equation}\label{beh.Q6}
f(\xi_0)=0, \qquad f(\xi)>0 \ {\rm for} \ \xi\in(\xi_0,\xi_0+\delta), \qquad (f^m)'(\xi_0)>0.
\end{equation}
\end{lemma}
Let us remark that the profiles as in \eqref{beh.Q5} and \eqref{beh.Q6} do not give rise to weak solutions in the form \eqref{SSS} to Eq. \eqref{eq1} since the contact condition $(f^m)'(\xi_0)=0$ (see \cite[Section 9.8]{VPME}) is not fulfilled at the edge of the support, but to \emph{subsolutions} (if we extend them by zero either before or after $\xi=\xi_0$).
\begin{proof}
The fact that $Q_5$ is a stable node and that $Q_6$ is an unstable node follow readily from the analysis of the linearization of the system \eqref{systinf2} (with the above mentioned choice of the sign) in a neighborhood of its origin. For the local behavior, the analysis is more tedious but follows closely the calculations performed in \cite[Lemma 2.6]{IS21} (see also \cite[Lemma 3.3]{IS24}). We omit the details.
\end{proof}
Finally, we are left with the critical point $Q_4$. Instead of performing a complete study of this point, we will just need to know that it cannot be reached by any trajectory arriving from the negative half-space $\{y<0\}$ in the system \eqref{PSsyst2}. To this end, we recall that trajectories reaching $Q_4$ have to fulfill the limits
$$
Z(\eta)\to\infty, \qquad \frac{X(\eta)}{Z(\eta)}\to0, \qquad \frac{Y(\eta)}{Z(\eta)}\to0, \qquad {\rm as} \ \eta\to\infty,
$$
which is completely equivalent, by \eqref{PSchange.inf}, to
\begin{equation}\label{loc.Q4}
z(\eta_1)\to\infty, \qquad \frac{z(\eta_1)}{x(\eta_1)}\to\infty, \qquad \frac{y(\eta_1)}{z(\eta_1)}\to0, \qquad {\rm as} \ \eta_1\to\infty.
\end{equation}
\begin{lemma}\label{lem.Q4}
There is no trajectory of the system \eqref{PSsyst2} entering the critical point $Q_4$ from the region $\{(x,y,z)\in\real^3: x>0, z>0, y<0\}$.
\end{lemma}
\begin{proof}
Assume for contradiction that there is such a trajectory $(x,y,z)(\eta_1)$ and some $\eta_{1,*}\in\real$ such that
$$
(x(\eta_1),y(\eta_1),z(\eta_1))\in\{(x,y,z)\in\real^3: x>0, z>0, y<0\}, \qquad \eta_1\in(\eta_{1,*},\infty)
$$
and that $(x,y,z)(\eta_1)$ has $Q_4$ as $\omega$-limit as $\eta_1\to\infty$. We infer from \eqref{loc.Q4} and by undoing the change of variable \eqref{PSchange2} that such trajectories correspond to profiles such that
\begin{equation}\label{interm6}
\xi^{\sigma}f(\xi)^{p-1}\to\infty, \qquad \xi^{\sigma+2}f(\xi)^{p-m}\to\infty, \qquad \frac{f'(\xi)}{\xi^{\sigma+1}f(\xi)^{p-m+1}}\to0,
\end{equation}
as $\xi\to\infty$, with $f$ decreasing on $(\xi_{*},\infty)$, $\xi_{*}=e^{\eta_{1,*}}$ by the definition of $y$ in \eqref{PSchange2}. In particular, there exists
$$
L_{\infty}=\lim\limits_{\xi\to\infty}f(\xi)\in[0,\infty).
$$
We can further write the equation \eqref{SSODE} in the form
\begin{equation}\label{interm7}
(f^m)''(\xi)-\frac{1}{2}\xi^{\sigma}f(\xi)^p+\frac{N-1}{\xi}(f^m)'(\xi)+\beta\xi f'(\xi)+f(\xi)\left[\alpha-\frac{1}{2}\xi^{\sigma}f(\xi)^{p-1}\right]=0,
\end{equation}
and we infer from \eqref{interm7} by dividing by $f(\xi)$, taking into account that $f'(\xi)<0$, $(f^m)'(\xi)<0$ for $\xi>\xi_{*}$ and using \eqref{interm6} that
\begin{equation}\label{interm8}
\lim\limits_{\xi\to\infty}\frac{(f^m)''(\xi)}{f(\xi)}=+\infty.
\end{equation}
If $L_{\infty}>0$, \eqref{interm8} gives that $(f^m)''(\xi)\to\infty$ as $\xi\to\infty$, which is a contradiction with the fact that $f$ is decreasing on $(\xi_{*},\infty)$. If $L_{\infty}=0$, the definition of the limit entails that, for any $K>0$, there is $\xi(K)>\xi_{*}>0$ such that $(f^m)''(\xi)>Kf(\xi)$, for $\xi>\xi(K)$. By multiplying the previous estimate by $(f^m)'(\xi)<0$, integrating over $(\xi_0,\xi)\subset(\xi(K),\infty)$ and recalling the monotonicity of $f(\xi)$, we find
$$
[(f^m)']^2(\xi)-[(f^m)']^2(\xi_0)\leq\frac{Km}{m+1}(f^{m+1}(\xi)-f^{m+1}(\xi_0))<0, \qquad \xi>\xi_0>\xi(K).
$$
By changing signs in the previous estimate, we get
$$
0<\frac{Km}{m+1}(f^{m+1}(\xi_0)-f^{m+1}(\xi))\leq[(f^m)']^2(\xi_0)-[(f^m)']^2(\xi)<[(f^m)']^2(\xi_0),
$$
and by letting $\xi\to\infty$ and recalling that $f(\xi)\to0$ we find, after taking square roots,
$$
\sqrt{\frac{Km}{m+1}}f(\xi_0)^{(m+1)/2}\leq|(f^m)'(\xi_0)|=-(f^m)'(\xi_0), \qquad \xi_0>\xi(K).
$$
Relabeling $\xi_0$ by $\xi$, we further deduce that
$$
\sqrt{\frac{K}{m(m+1)}}\leq-f^{(m-3)/2}(\xi)f'(\xi)=\frac{2}{m-1}\left|(f^{(m-1)/2})'(\xi)\right|,
$$
for any $\xi>\xi(K)$. Since $K>0$ has been chosen arbitrarily in the previous estimates and $m>1$, we get that
$$
\lim\limits_{\xi\to\infty}(f^{(m-1)/2})'(\xi)=-\infty,
$$
which is a contradiction with the fact that $f^{(m-1)/2}$ is a positive function decreasing to zero as $\xi\to\infty$. This contradiction gives that there is no such trajectory as assumed at the beginning, completing the proof.
\end{proof}
We are now ready to proceed with the global analysis of the system, leading to the proof of the main theorems.

\section{Some preparatory results of global analysis}\label{sec.prep}

We gather in this section some important preparatory results concerning the global analysis of the trajectories of the system \eqref{PSsyst2}, needed in the proofs of the main theorems. The first one establishes a positively invariant region which will play a very significant role in the forthcoming analysis.
\begin{lemma}\label{lem.invar}
The region
$$
\mathcal{R}:=\{(x,y,z)\in\real^3: y>0, z>x\}
$$
is positively invariant for the system \eqref{PSsyst2}: that is, if for a trajectory, there is $\eta_{1,*}\in\real$ such that $(x,y,z)(\eta_{1,*})\in\mathcal{R}$, then $(x,y,z)(\eta_1)\in\mathcal{R}$ for any $\eta_1>\eta_{1,*}$.
\end{lemma}
\begin{proof}
The flow of the system \eqref{PSsyst2} across the plane $\{y=0\}$ (with normal vector $(0,1,0)$) has the direction given by the sign of the expression $z-x>0$ in $\mathcal{R}$. The flow of the system \eqref{PSsyst2} across the plane $z=x$ (with normal vector $(-1,0,1)$) has the direction given by the sign of the expression
$$
f(y,z)=z(\sigma+2+(p-m)y)-z(2-(m-1)y)=z(\sigma+(p-1)y)>0,
$$
in $\mathcal{R}$. Thus, a trajectory passing through a point in $\mathcal{R}$ cannot leave this region through none of its two ``walls" (the planes $\{y=0\}$ and $\{z=x\}$) and will remain there forever, as claimed.
\end{proof}
An immediate consequence of this lemma is the behavior of the trajectory $l_{\infty}$, that is, the unique trajectory on the unstable manifold of the critical point $Q_1$ contained in the plane $\{x=0\}$.
\begin{lemma}\label{lem.x0}
The trajectory $l_{\infty}$ enters and remains in the region $\mathcal{R}$.
\end{lemma}
\begin{proof}
It is obvious that $z>x=0$ along this trajectory. Moreover, the system \eqref{PSsyst2} reduces in the invariant plane $\{x=0\}$ to
\begin{equation}\label{PSsyst2x0}
\left\{\begin{array}{ll}\dot{y}=-(N-2)y+z-my^2, \\ \dot{z}=z(\sigma+2+(p-m)y),\end{array}\right.
\end{equation}
and the flow of the system \eqref{PSsyst2x0} across the axis $\{y=0\}$ is given by the sign of $z$, which is always non-negative. Thus, the trajectory $l_{\infty}$ goes out tangent to the eigenvector $e_3=(0,1,N+\sigma)$ of the matrix $M(Q_1)$ given in Lemma \ref{lem.Q1}, entering the region $\mathcal{R}$ and thus remaining there afterwards, according to Lemma \ref{lem.invar}.
\end{proof}
The next lemma establishes the global behavior of the trajectory $l_0$, which is rather interesting and gives us a clear understanding of how the exponent $p_F(\sigma)$ comes into play with a decisive role in the classification.
\begin{lemma}\label{lem.z0}
The trajectory $l_0$, corresponding to taking $C=0$ in \eqref{lC} and contained in the plane $\{z=0\}$, has the following properties:

(a) If $m<p<p_F(\sigma)$, it connects to the critical point $Q_5$.

(b) If $p=p_F(\sigma)$, it is explicit and connects to the critical point $P_2$.

(c) If $p>p_F(\sigma)$, it connects to the critical point $P_1$.
\end{lemma}
\begin{proof}
The system \eqref{PSsyst2} reduces in the invariant plane $\{z=0\}$ to
\begin{equation}\label{PSsyst2z0}
\left\{\begin{array}{ll}\dot{x}=x(2-(m-1)y), \\ \dot{y}=-x-(N-2)y-my^2-\frac{p-m}{\sigma+2}xy.\end{array}\right.
\end{equation}
The orbit $l_0$ goes out of $Q_1$ tangent to the eigenvector $e_1=(N,-1,0)$, as established in Lemma \ref{lem.Q1}, and thus enters the half-plane $\{y<0\}$ and remains there forever, since the flow of the system \eqref{PSsyst2z0} across the axis $\{y=0\}$ points into the negative direction (as the sign of $-x$). We thus infer from the first equation in \eqref{PSsyst2z0} that $\eta_1\mapsto x(\eta_1)$ is an increasing function along $l_0$, hence we can invert this mapping and thus express the trajectory $l_0$ as a graph $y=y(x)$, such that, by the inverse function theorem,
\begin{equation}\label{interm4}
\frac{dy}{dx}=\frac{-x-(N-2)y(x)-my(x)^2-[(p-m)/(\sigma+2)]xy(x)}{x(2-(m-1)y(x))}.
\end{equation}
A direct and easy calculation shows that $y(x)=-x/N$ satisfies \eqref{interm4} exactly when $p=p_F(\sigma)$, that is, the trajectory $l_0$ is the line $y=-x/N$. We deduce from \eqref{PSchange.inf} that, in $(X,Y,Z)$ variables, the line $y=-x/N$ is seen as
$$
Y=\frac{y}{x}=-\frac{1}{N}=-\frac{p-m}{\sigma+2}=-\frac{\beta}{\alpha},
$$
so that it ends at the point $P_2$, completing the proof of part (b).

\medskip

Assume now that $p>p_F(\sigma)$, which is equivalent to $\beta/\alpha>1/N$. The flow of the system \eqref{PSsyst2z0} across the line
$$
(r_0): \qquad \left\{y=-\frac{(p-m)x}{\sigma+2}\right\}, \qquad {\rm with \ normal} \qquad \overline{n}=\left(\frac{p-m}{\sigma+2},1\right),
$$
is given by the sign of the expression
\begin{equation}\label{interm5}
F(x)=\frac{Nx(p-p_F(\sigma))}{\sigma+2}>0.
\end{equation}
Since on the trajectory $l_0$ we have
$$
y(\eta_1)\sim-\frac{x(\eta_1)}{N}>-\frac{(p-m)x(\eta_1)}{\sigma+2}, \qquad {\rm as} \ \eta_1\to-\infty,
$$
we infer that $l_0$ goes out from $Q_1$ into the region limited by the line $r_0$ and the $x$ axis, and it will stay forever in this region, according to \eqref{interm5}. Passing to the $(X,Y,Z)$ variables by undoing \eqref{PSchange.inf}, this region is seen as the strip
$$
\mathcal{S}=\left\{(X,Y)\in\real^2: -\frac{\beta}{\alpha}<Y<0\right\}
$$
of the invariant plane $\{Z=0\}$. The first equation of the system \eqref{PSsyst2z0} establishes that the coordinate $x$ is increasing on the trajectory $l_0$, or equivalently, $X=1/x$ decreases, hence, there is $X_{\infty}=\lim\limits_{\eta\to\infty}X(\eta)\geq0$. Since $Y(\eta)$ is bounded in the strip $\mathcal{S}$, we readily infer from the Poincar\'e-Bendixon theory \cite[Section 3.7]{Pe} that the trajectory should end at a critical point, as there are obviously no periodic orbits with $\eta\mapsto X(\eta)$ monotone. Thus, it either connects to $P_1$ or to $P_2$. But the unique trajectory entering $P_2$ on the stable manifold of it, inside the plane $\{Z=0\}$, arrives tangent to the eigenvector
$$
e_1=\left(1,\frac{N\beta-\alpha}{m\beta}\right),
$$
according to Lemma \ref{lem.P2}. Since $N\beta-\alpha>0$, it follows that this trajectory enters $P_2$ from the region $\{Y<-\beta/\alpha\}$, that is, outside the strip $\mathcal{S}$, hence $l_0$ cannot reach $P_2$ and consequently, it will arrive to the asymptotically stable point $P_1$, proving part (c).

\medskip

Let now $m<p<p_F(\sigma)$, that is, $\beta/\alpha<1/N$. Noticing that now $F(x)<0$ in \eqref{interm5}, similar arguments as in the previous proof establish that, in this case, the trajectory $l_0$, seen in $(X,Y,Z)$ variables, will enter and stay in the half-space $\{Y<-\beta/\alpha\}$, while now $N\beta-\alpha<0$, which means that the unique trajectory entering $P_2$ comes from the interior of the strip $\mathcal{S}$. We conclude that, once more, $l_0$ does not arrive to $P_2$. Moreover, the first equation in \eqref{PSsyst1} implies that $\eta\mapsto X(\eta)$ is decreasing on $l_0$, so that $X(\eta)\to X_{\infty}\in[0,\infty)$ as $\eta\to\infty$. The Poincar\'e-Bendixon theory then readily implies that $Y(\eta)\to-\infty$ as $\eta\to\infty$, since there is no finite critical point in the half-plane $\{Y\leq-\beta/\alpha\}$ (except $P_2$, that was discarded above). Hence, on $l_0$ we have $Y(\eta)\to-\infty$, $Y(\eta)/X(\eta)\to-\infty$, as $\eta\to\infty$, which shows that the limit is the stable node $Q_5$, proving part (a).
\end{proof}
This lemma is very important in the forthcoming proofs, since it shows how the behavior of one of the limits of the unstable manifold of $Q_1$ changes when $p=p_F(\sigma)$, which is the reason for which there is a strong difference between the outcome of Theorem \ref{th.small} with respect to Theorem \ref{th.large}. We conclude this section with one more technical result needed in the proofs of the theorems.
\begin{lemma}\label{lem.region}
Let $(x,y,z)(\eta_1)$ be a trajectory of the system \eqref{PSsyst2} such that there is $\eta_{1,*}\in\real$ with the property that $x(\eta_{1,*})>0$, $z(\eta_{1,*})>0$ and
\begin{equation}\label{cond.region}
-\frac{\sigma+2}{p-m}<y(\eta_1)<0, \qquad {\rm for \ any} \ \eta_1>\eta_{1,*}.
\end{equation}
Then, this trajectory ends by connecting to one of the critical points $P_{\gamma_0}$ or $P_1$.
\end{lemma}
\begin{proof}
The condition \eqref{cond.region} together with the first and third equation of \eqref{PSsyst2} show that $\eta_1\mapsto x(\eta_1)$, $\eta_1\mapsto z(\eta_1)$ are increasing functions on $(\eta_{1,*},\infty)$ on the trajectory under consideration. Thus, there exist
$$
x_{\infty}:=\lim\limits_{\eta_1\to\infty}x(\eta_1)>0, \qquad z_{\infty}:=\lim\limits_{\eta_1\to\infty}z(\eta_1)>0.
$$
Assume for contradiction that $x_{\infty}<\infty$. If also $z_{\infty}<\infty$, similar arguments as in the proof of \cite[Proposition 4.10]{ILS24} entail that the $\omega$-limit of the trajectory has to be a (finite) critical point, and there is no such point satisfying \eqref{cond.region} with $x>0$, $z>0$. Thus, $z_{\infty}=\infty$, and since $x_{\infty}<\infty$ and $y(\eta_1)$ is bounded by \eqref{cond.region}, we conclude that the trajectory ends at the critical point $Q_4$, which contradicts Lemma \ref{lem.Q4}, since the trajectory would reach $Q_4$ coming from the half-space $\{y<0\}$. We thus deduce that $x_{\infty}=\infty$ and in particular $(y/x)(\eta_1)\to0$ as $\eta_1\to\infty$ by \eqref{cond.region}. Passing to $(X,Y,Z)$ variables by undoing \eqref{PSchange.inf}, we find that $X(\eta)\to0$ and $Y(\eta)\to0$ along this trajectory, as $\eta\to\infty$. We thus obtain a trajectory of the system \eqref{PSsyst1} having an $\omega$-limit set included in the $Z$ axis. Lemmas \ref{lem.P1} and \ref{lem.Pg} show that this limit is either one of the critical points $P_1$ or $P_{\gamma_0}$, or a segment of the critical line $\{X=0, Y=0\}$ of the system \eqref{PSsyst1} which cannot have an endpoint at zero due to the stability of $P_1$. The latter conclusion is equivalent to the corresponding profile oscillating between two hyperbolas
$$
A_1\xi^{-\sigma/(p-1)}\leq f(\xi)\leq A_2\xi^{-\sigma/(p-1)}, \qquad \xi\geq\xi_0>0,
$$
for some constants $0<A_1<A_2<\infty$ and some $\xi_0>0$ very large. Letting then $g(\xi):=\xi^{\sigma/(p-1)}f(\xi)$, we find by direct calculation that
\begin{equation}\label{ode.g}
\begin{split}
\xi^2(g^m)''(\xi)&-\left(\frac{2m\sigma}{p-1}-N+1\right)\xi(g^m)'(\xi)+\frac{m\sigma}{p-1}\left(\frac{m\sigma}{p-1}-N+2\right)g^m(\xi)\\
&+\xi^{L/(p-1)}\left[\frac{1}{p-1}g(\xi)+\beta\xi g'(\xi)-g^p(\xi)\right]=0.
\end{split}
\end{equation}
Let $(\xi^{m}_{k})_{k\geq1}$, respectively $(\xi^{M}_{k})_{k\geq1}$ be two sequences of local minima of $g$, respectively local maxima of $g$, such that $\xi^{m}_{k}\to\infty$, $\xi^{M}_{k}\to\infty$ as $k\to\infty$, and
$$
g(\xi^m_k)\to L_{{\rm \inf}}:=\liminf\limits_{\xi\to\infty}g(\xi)\in[A_1,A_2], \quad g(\xi^M_k)\to L_{{\rm sup}}:=\limsup\limits_{\xi\to\infty}g(\xi)\in[A_1,A_2].
$$
Evaluating \eqref{ode.g} at $\xi=\xi^m_k$, respectively at $\xi=\xi^M_{k}$, and taking into account that $g^m(\xi^m_{k})$, respectively $g^m(\xi^M_{k})$, are bounded, we readily deduce that the big term in brackets in \eqref{ode.g} has to compensate, for $k\geq1$ sufficiently large, the term $\xi^2(g^m)''(\xi)$ on these sequences. We thus have
$$
\lim\limits_{k\to\infty}\left[\frac{1}{p-1}g(\xi^m_{k})-g^p(\xi^m_{k})\right]\leq0, \quad \lim\limits_{k\to\infty}\left[\frac{1}{p-1}g(\xi^M_{k})-g^p(\xi^M_{k})\right]\geq0,
$$
whence
$$
L_{{\rm inf}}\geq\left(\frac{1}{p-1}\right)^{1/(p-1)}\geq L_{{\rm sup}}.
$$
This implies that both limits are equal to the constant $(1/(p-1))^{1/(p-1)}$ and the trajectory enters the critical point $P_{\gamma_0}$. This leads to a contradiction with the possibility of infinite, non-damped oscillations and, thus, to the conclusion that the trajectory ends at one of the critical points $P_1$ or $P_{\gamma}$, as stated.
\end{proof}

\section{Monotonicity}\label{sec.monot}

In this section, we prove that the profiles $f(\cdot;A)$ solutions to the Cauchy problem \eqref{SSODE}-\eqref{IC} are ordered while they are decreasing. This fact, which paves the way towards uniqueness of the compactly supported very singular solution, is proved by employing a \emph{sliding technique} which stems, up to our knowledge, from the classical paper \cite{FK80} (and in the form that follows, from \cite{YeYin}), but has been employed with success by one of the authors and his collaborators in recent works such as \cite{IMS23, ILS24}. In order to employ this method, we first need a local behavior near $\xi=0$ of the profiles $f(\cdot;A)$ more precise than the one given in Lemma \ref{lem.Q1}. We follow at this point the ideas in \cite[Section 4.1]{ILS24}. Let us consider thus the following Cauchy problem associated to the part of \eqref{SSODE} coming only from the porous medium equation:
\begin{equation}\label{PMEODE}
(\phi^m)''(\xi)+\frac{N-1}{\xi}(\phi^m)'(\xi)+\alpha\phi(\xi)+\beta\xi\phi'(\xi)=0,
\end{equation}
with initial conditions $\phi(0)=A$, $\phi'(0)=0$, and denote by $\phi(\cdot;A)$ its (unique) solution. The following result shows that a number of the first terms in the Taylor expansion of the profile $f(\cdot;A)$ are similar to the ones of the Taylor expansion of the profile $\phi(\cdot;A)$. The following basic property (but whose proof is rather technical) has been published as \cite[Lemma 4.2]{ILS24}.
\begin{lemma}\label{lem.exp1}
For $A>0$ and for any integer $k$ such that $2\leq k<2+\sigma$, we have
$$
f(\xi;A)-\phi(\xi;A)=o(\xi^k), \qquad f^m(\xi;A)-\phi^m(\xi;A)=o(\xi^k), \qquad {\rm as} \ \xi\to0.
$$
\end{lemma}
This result has been proved in \cite[Lemma 4.2]{ILS24} under the hypothesis $p\in(0,1)$. An inspection of the proof therein shows that the last term $\xi^{\sigma}f^p(\xi)$ has absolutely no influence in the calculations, thus the proof is completely identical and we omit it here. We next introduce the first order in the expansion as $\xi\to0$ which depends on $\sigma$, following \cite[Lemma 4.3]{ILS24}.
\begin{lemma}\label{lem.exp2}
Let $k_0$ be the largest integer strictly below $\sigma$ and let $A\in(0,\infty)$. Then, as $\xi\to0$,
\begin{equation}\label{exp.smallF1}
f^m(\xi;A)=\sum\limits_{j=0}^{k_0+2}B_j\xi^j+\frac{A^p}{(\sigma+2)(\sigma+N)}\xi^{\sigma+2}+o(\xi^{\sigma+2}),
\end{equation}
if $\sigma\not\in\mathbb{N}$ and $k_0$ is the integer part of $\sigma$, or
\begin{equation}\label{exp.smallF2}
f^m(\xi;A)=\sum\limits_{j=0}^{k_0+3}B_j\xi^j+\frac{A^p}{(\sigma+2)(\sigma+N)}\xi^{\sigma+2}+o(\xi^{\sigma+2}),
\end{equation}
if $\sigma\in\mathbb{N}$ and $k_0=\sigma-1$, where $B_j$ are the Taylor coefficients of the expansion of the function $\phi^m(\cdot;A)$, with $\phi(\cdot;A)$ solution to \eqref{PMEODE} with initial conditions $\phi(0)=A$, $\phi'(0)=0$.
\end{lemma}
\begin{proof}
This is an immediate adaptation of the proof of \cite[Lemma 4.3]{ILS24}, but we give a sketch of it due to its importance for the forthcoming monotonicity result. Let us introduce the function
$$
H(\xi;A)=\xi^{N-1}(f^m)'(\xi;A)+\beta\xi^N f(\xi;A).
$$
We readily observe that
\begin{equation*}
\begin{split}
H'(\xi;A)&=\xi^{N-1}\left[(f^m)''(\xi;A)+\frac{N-1}{\xi}(f^m)'(\xi;A)+\beta\xi f'(\xi;A)\right]+N\beta\xi^{N-1}f(\xi;A)\\
&=(N\beta-\alpha)\xi^{N-1}f(\xi;A)+\xi^{N+\sigma-1}f^p(\xi;A).
\end{split}
\end{equation*}
Let us now restrict ourselves to the case $\sigma\not\in\mathbb{N}$ and let $k_0$ be the integer part of $\sigma$ (the other case being very similar). We infer from Lemma \ref{lem.exp1}, \eqref{beh.Q1} and the previous calculation that
\begin{equation*}
\begin{split}
H'(\xi;A)&=(N\beta-\alpha)\sum\limits_{j=0}^{k_0+2}b_j\xi^{N+j-1}+o(\xi^{N+k_0+1})\\&+\xi^{N+\sigma-1}\left[A^{m-1}-\frac{\alpha(m-1)}{2mN}\xi^2+o(\xi^2)\right]^{p/(m-1)}\\
&=(N\beta-\alpha)\sum\limits_{j=0}^{k_0+2}b_j\xi^{N+j-1}+o(\xi^{N+k_0+1})+A^p\xi^{N+\sigma-1}+o(\xi^{N+\sigma})\\
&=(N\beta-\alpha)\sum\limits_{j=0}^{k_0}b_j\xi^{N+j-1}+A^p\xi^{N+\sigma-1}+o(\xi^{N+\sigma-1}),
\end{split}
\end{equation*}
where $b_j$ are the coefficients of the expansion of the function $\phi(\cdot;A)$. We further get by integration that
$$
H(\xi;A)=(N\beta-\alpha)\sum\limits_{j=0}^{k_0}\frac{b_j}{N+j}\xi^{N+j}+\frac{A^p}{N+\sigma}\xi^{N+\sigma}+o(\xi^{N+\sigma}).
$$
Recalling the definition of $H(\xi;A)$, we obtain after easy manipulations the expansion of $(f^m)'(\xi;A)$ as follows
$$
(f^m)'(\xi;A)=\sum\limits_{j=0}^{k_0}\left(-\beta+\frac{N\beta-\alpha}{N+j}\right)b_j\xi^{j+1}+\frac{A^p}{N+\sigma}\xi^{\sigma+1}+o(\xi^{\sigma+1}),
$$
which leads to \eqref{exp.smallF1} by one more integration step. The calculation is completely analogous when $\sigma\in\mathbb{N}$, and we refer the reader to \cite[Lemma 4.3]{ILS24} for the details.
\end{proof}
We are now in a position to prove the monotonicity lemma for decreasing profiles. More precisely, we have
\begin{lemma}\label{lem.monot}
Let $0<A_1<A_2<\infty$, $f_1=f(\cdot;A_1)$, respectively $f_2=f(\cdot;A_2)$ and let $\Xi\in(0,\infty)$ such that $f_1(\xi)>0$, $f_1'(\xi)<0$, $f_2'(\xi)<0$ for any $\xi\in(0,\Xi)$. Then $f_1(\xi)<f_2(\xi)$ for any $\xi\in(0,\Xi)$.
\end{lemma}
\begin{proof}
Let us denote $g_i=f_i^m$, $i=1,2$, hence $g_i$ is a solution to
\begin{equation}\label{ODE2}
g''(\xi)+\frac{N-1}{\xi}g'(\xi)+\alpha g^{1/m}(\xi)+\beta\xi(g^{1/m})'(\xi)-\xi^{\sigma}g^{p/m}(\xi)=0.
\end{equation}
We introduce the following rescaling, that will be useful in the sequel:
\begin{equation}\label{resc}
f_{\lambda}(\xi):=\lambda^{-2/(m-1)}f_1(\lambda\xi), \qquad g_{\lambda}(\xi):=\lambda^{-2m/(m-1)}g_1(\lambda\xi).
\end{equation}
Straightforward calculations then imply that $g_{\lambda}$ is a solution to the differential equation
\begin{equation}\label{ODEresc}
g_{\lambda}''(\xi)+\frac{N-1}{\xi}g_{\lambda}'(\xi)+\alpha g_{\lambda}(\xi)^{1/m}+\beta\xi(g_{\lambda}^{1/m})'(\xi)-\lambda^{L/(m-1)}\xi^{\sigma}g_{\lambda}^{p/m}(\xi)=0.
\end{equation}
where we recall that $L=\sigma(m-1)+2(p-1)>0$. Since $A_1<A_2$, it follows that $f_1(\xi)<f_2(\xi)$ in a right neighborhood of $\xi=0$. Assume for contradiction that there is $\xi_0\in(0,\Xi)$ such that $f_1(\xi_0)=f_2(\xi_0)$, that is, also $g_1(\xi_0)=g_2(\xi_0)$. Let us observe at this point that, if $0<\lambda<\lambda'\leq 1$, then the monotonicity of $g_1$ on $[0,\Xi]$ entails on the one hand that
$$
g_1(\lambda'\xi)<g_1(\lambda\xi), \qquad \xi\in[0,\Xi],
$$
therefore,
\begin{equation*}
\begin{split}
g_{\lambda}(\xi)&=\lambda^{-2m/(m-1)}g_1(\lambda\xi)>(\lambda')^{-2m/(m-1)}g_1(\lambda\xi)\\
&>(\lambda')^{-2m/(m-1)}g_1(\lambda'\xi)=g_{\lambda'}(\xi),
\end{split}
\end{equation*}
for any $\xi\in[0,\Xi]$. On the other hand,
\begin{equation}\label{interm9}
\lim\limits_{\lambda\to0}\min\limits_{[0,\Xi]}g_{\lambda}=\lim\limits_{\lambda\to 0}g_{\lambda}(\Xi)=\lim\limits_{\lambda\to0} \lambda^{-2m/(m-1)} g_1(\lambda\Xi)=\infty,
\end{equation}
thus we can introduce the optimal sliding parameter
\begin{equation}\label{interm10}
\lambda_0:=\sup\{\lambda\in(0,1): g_2(\xi)<g_{\lambda}(\xi), \xi\in[0,\xi_0]\}
\end{equation}
and infer from \eqref{interm9} and the ordering $g_1(\xi)<g_2(\xi)$ for $\xi\in(0,\xi_0)$ that $\lambda_0\in(0,1)$. We further deduce from the definition of $\lambda_0$ that $g_2(\xi)\leq g_{\lambda_0}(\xi)$ for any $\xi\in[0,\xi_0]$ and that there exists some contact point $\xi_1\in[0,\xi_0]$ such that $g_2(\xi_1)=g_{\lambda_0}(\xi_1)$ (otherwise we reach an immediate contradiction with the optimality of $\lambda_0$, since $g_{\lambda_0}-g_2>0$ on the compact set $[0,\xi_0]$). We split the rest of the proof into three cases.

\medskip

\noindent \textbf{Case 1: $\xi_1=\xi_0$}. We then have, owing to the monotonicity of $g_1$ on $[0,\xi_0]$ and the fact that $\lambda_0<1$,
$$
g_1(\xi_0)=g_2(\xi_0)=g_{\lambda_0}(\xi_0)=\lambda_0^{-2m/(m-1)}g_1(\lambda_0\xi_0)<g_1(\lambda_0\xi_0)<g_1(\xi_0),
$$
which is a contradiction.

\medskip

\noindent \textbf{Case 2: $\xi_1\in(0,\xi_0)$}. In this case, we have
$$
g_2(\xi_1)=g_{\lambda_0}(\xi_1), \qquad g_2'(\xi_1)=g_{\lambda_0}'(\xi_1), \qquad g_2''(\xi_1)\leq g_{\lambda_0}''(\xi_1).
$$
Introducing the previous relations into both \eqref{ODE2} solved by $g_2$ and \eqref{ODEresc} solved by $g_{\lambda_0}$ and subtracting these equalities, we obtain
$$
g_{\lambda_0}''(\xi_1)-g_2''(\xi_1)=\xi_1^{\sigma}(\lambda_0^{L/(m-1)}-1)g_2(\xi_1)<0,
$$
which is a contradiction with the fact that $g_2''(\xi_1)\leq g_{\lambda_0}''(\xi_1)$.

\medskip

\noindent \textbf{Case 3: $\xi_1=0$}. It then follows, on the one hand, that $g_2(0)=g_{\lambda_0}(0)=A_2^m$, and we infer from \eqref{resc} that
\begin{equation}\label{interm11}
A_2=A_1\lambda_0^{-2/(m-1)}.
\end{equation}
On the other hand, we know that $g_2(\xi)<g_{\lambda_0}(\xi)$ for any $\xi\in(0,\xi_0)$. Taking into account the expansions as $\xi\to0$ given in Lemma \ref{lem.exp2} for, respectively, the functions $g_2(\xi)$ and $g_{\lambda_0}(\xi)$, we readily observe that the terms depending only on the expansion coming from the solution to \eqref{PMEODE} are mapped one into the other in view of the fact that the porous medium equation is invariant to the rescaling \eqref{resc}, and the first order where a difference is seen is the first one involving $\sigma$, exposed in Lemma \ref{lem.exp2}. We thus have, after also employing \eqref{interm11}, as $\xi\to0$,
\begin{equation*}
\begin{split}
g_2(\xi)-g_{\lambda_0}(\xi)&=\frac{A_2^p}{(\sigma+2)(\sigma+N)}\xi^{\sigma+2}-
\frac{A_1^p\lambda_0^{-2m/(m-1)}}{(\sigma+2)(\sigma+N)}(\lambda_0\xi)^{\sigma+2}+o(\xi^{\sigma+2})\\
&=\frac{A_2^p}{(\sigma+2)(\sigma+N)}\left[1-\lambda_0^{L/(m-1)}\right]\xi^{\sigma+2}+o(\xi^{\sigma+2})>0,
\end{split}
\end{equation*}
since $L/(m-1)>0$ and $\lambda_0\in(0,1)$. But the latter is a contradiction with the fact that $g_2(\xi)<g_{\lambda_0}(\xi)$ in a right neighborhood of the origin. Reaching a contradiction also in this last case completes the proof.
\end{proof}

\section{Proofs of the main results}\label{sec.thm}

We are now ready to proceed with the proofs of Theorems \ref{th.gen}, \ref{th.large} and \ref{th.small}. Recalling the local analysis performed in Section \ref{sec.transf}, and specially Lemmas \ref{lem.P1}, \ref{lem.P2}, \ref{lem.Pg} and \ref{lem.Q1}, the proofs are based on a shooting method from the critical point $Q_1$ on the trajectories $l_C$ defined in \eqref{lC} for $C\in(0,\infty)$. We recall here that these trajectories are in a one-to-one and onto correspondence with profiles $f(\cdot;A)$ with $A\in(0,\infty)$, the correspondence being given in \eqref{bij}. Thus, we will be looking for shooting parameters $C\in(0,\infty)$ such that $l_C$ either connects to the critical point $P_{\gamma_0}$ leading to the decay \eqref{decay.Pg} as $\xi\to\infty$, or to the critical point $P_1$ leading to the decay \eqref{decay.P1}, or finally to the critical point $P_2$ leading to compactly supported profiles (and very singular solutions, as explained in the Introduction). For the easiness of the reading, we split the proofs into several subsections.

\subsection{Proof of Theorems \ref{th.gen} and \ref{th.large} for $p\geq p_F(\sigma)$: existence}\label{sec.gen}

Let us fix throughout this section $p\geq p_F(\sigma)$. We need one more preparatory lemma, before proceeding with the proof.
\begin{lemma}\label{lem.barlarge}
Let $p\geq p_F(\sigma)$ and $(X,Y,Z)(\eta)$ be a trajectory of the system \eqref{PSsyst1} such that there is $\eta_*\in\real$ with $X(\eta_*)>0$, $Z(\eta_*)>0$ and
\begin{equation}\label{strip}
-\frac{\beta}{\alpha}<Y(\eta)<0, \qquad {\rm for \ any} \ \eta\in(\eta_*,\infty).
\end{equation}
Then the trajectory ends either at the critical point $P_1$ or at the critical point $P_{\gamma}$. Moreover, the plane $\{Y=-\beta/\alpha\}$ cannot be crossed towards the negative side by trajectories of the system \eqref{PSsyst1}.
\end{lemma}
\begin{proof}
We infer from the first equation of \eqref{PSsyst1} that $\eta\mapsto X(\eta)$ is decreasing along the trajectory for $\eta>\eta_*$, hence there exists a limit $X_0=\lim\limits_{\eta\to\infty}X(\eta)\geq0$. It follows that the $\omega$-limit set of the trajectory (which is either a critical point or not) lies in the plane $\{X=X_0\}$. If $X_0>0$, the $\omega$-limit set cannot be a critical point (as there is none with $X=X_0$) and, as an orbit, it is itself a trajectory of the system (according to, for example, \cite[Theorem 2, Section 3.2]{Pe}). But this, together with the first equation, gives $(m-1)X_0-2Y=0$ in this $\omega$-limit set, which contradicts the fact that $Y(\eta)<0$ for $\eta\geq\eta_*$. It thus follows that $X_0=0$. Observing that the system \eqref{PSsyst1} reduces, in the invariant plane $\{X=0\}$, to
\begin{equation*}
\left\{\begin{array}{ll}\dot{Y}=-Y^2-\frac{p-m}{\sigma+2}Y,\\
\dot{Z}=(p-1)YZ,\end{array}\right.
\end{equation*}
which has no periodic orbits in the half-plane $\{Y<0\}$ due to the fact that $\dot{Z}=(p-1)YZ<0$, the Poincar\'e-Bendixon theory together with Lemma \ref{lem.Q4} and arguments as in the final part of the proof of Lemma \ref{lem.region} entail that the $\omega$-limit set of our trajectory must be a critical point among $P_1$, $P_{\gamma_0}$ or $P_2$. In the case $p>p_F(\sigma)$, an inspection of the eigenvectors $e_1$ and $e_3$ spanning the stable manifold of the critical point $P_2$ in Lemma \ref{lem.P2} show that the trajectories entering $P_2$ on its stable manifold (except the one belonging to the plane $\{X=0\}$) arrive through the half-space $\{Y<-\beta/\alpha\}$, contradicting \eqref{strip}. In the limit case $p=p_F(\sigma)$, we have to look for the second order of the Taylor expansion of the stable manifold of $P_2$. Indeed, if we set
$$
Y=-\frac{\beta}{\alpha}+aX+bZ+cX^2+dXZ+eZ^2+o(|(X,Z)|^2),
$$
and ask for the flow of the system to be of order $o(|(X,Z)|^2)$ across the previous surface, we find after some calculations $a=b=c=e=0$ and
$$
Y=-\frac{\beta}{\alpha}-\frac{N^2}{2mN-N+\sigma+2}XZ+o(|(X,Z)|^2),
$$
whence also in this case, the orbits entering $P_2$ arrive from the half-space $\{Y<-\beta/\alpha\}$. We infer that $P_2$ cannot be reached from the strip \eqref{strip}, completing the proof of the first statement. The last statement follows from the direction of the flow of the system \eqref{PSsyst1} across the plane $\{Y=-\beta/\alpha\}$ (with normal direction $(0,1,0)$), given by the sign of the expression
\begin{equation}\label{flowstrip}
F(X,Z)=X\left(Z+\frac{N(p-p_F(\sigma))}{\sigma+2}\right)>0.
\end{equation}
\end{proof}
We can now pass to the (rather simultaneous) proofs of \ref{th.gen} and \ref{th.large} in the range $p>p_F(\sigma)$.
\begin{proof}[Proof of Theorem \ref{th.large}: existence]
Let $p\geq p_F(\sigma)$. We split the interval $(0,\infty)$ in the following disjoint sets:
\begin{equation}\label{sets.large}
\begin{split}
&\mathcal{A}:=\{C\in(0,\infty): {\rm there \ is} \ \eta_{1,*}\in\real, \ y(\eta_{1,*})>0 \ {\rm on \ the \ trajectory} \ l_C\},\\
&\mathcal{C}:=\{C\in(0,\infty): y(\eta_1)<0 \ {\rm on} \ l_C \ {\rm for \ any} \ \eta_1\in\real \ {\rm and \ it \ ends \ at} \ P_1\},\\
&\mathcal{B}:=(0,\infty)\setminus(\mathcal{A}\cup\mathcal{C}).
\end{split}
\end{equation}
We readily find that $\mathcal{A}$ is an open set, by definition and continuity with respect to $C$. Similarly, the attracting stability of $P_1$ as established in Lemma \ref{lem.P1} entains that $\mathcal{C}$ is an open set. Since the orbits $l_C$ with $C>0$ go out from $Q_1$ into the region $\{y<0\}$ and the direction of the flow across the plane $\{y=0\}$ is given by the sign of $z-x$, it follows from Lemma \ref{lem.invar} that for any $C\in\mathcal{A}$, the trajectory $l_C$ enters the region $\mathcal{R}$ defined in Lemma \ref{lem.invar} and remains there. On the one hand, Lemma \ref{lem.x0} and the continuity with respect to $C$ then imply that $\mathcal{A}$ is non-empty, and more precisely there is $C^*>0$ such that $(C^*,\infty)\subseteq\mathcal{A}$. On the other hand, the non-emptyness of $\mathcal{C}$ follows from Lemma \ref{lem.z0}, the continuity with respect to $C$ and the stability of $P_1$ if $p>p_F(\sigma)$. In the limit case $p=p_F(\sigma)$, the trajectory $l_0$ enters $P_2$ by Lemma \ref{lem.z0}, but an application of the behavior near a saddle point (see for example \cite[Theorem 2.9]{Shilnikov}) to the saddle $P_2$ leads to the same conclusion as for $C>0$ small, the orbits $l_C$ follow the unstable manifold of $P_2$, which is contained in the $Y$ axis and approaches $P_1$.

We then infer that $\mathcal{B}\neq\emptyset$. Pick $C_0\in\mathcal{B}$. Since $C_0\not\in\mathcal{A}$, we deduce that $l_{C_0}$ is fully contained in the half-space $\{y\leq0\}$, that is, also in $\{Y\leq0\}$ according to \eqref{PSchange.inf}. Then, it has two possibilities:

$\bullet$ either $l_{C_0}$ is tangent to the plane $\{y=0\}$, that is, there is $\eta_{1,*}\in\real$ such that $y(\eta_{1,*})=0$, $\dot{y}(\eta_{1,*})=0$ and $y''(\eta_{1,*})\leq0$. From the system \eqref{PSsyst2} and these conditions, we infer that $x(\eta_{1,*})=z(\eta_{1,*})$ and furthermore, by taking derivatives in the second equation of \eqref{PSsyst2} and taking into account the previous equalities, we have
$$
y''(\eta_{1,*})=\dot{z}(\eta_{1,*})-\dot{x}(\eta_{1,*})=(\sigma+2)z(\eta_{1,*})-2x(\eta_{1,*})=\sigma z(\eta_{1,*})>0,
$$
which is a contradiction with the condition of maximum point $y''(\eta_{1,*})\leq0$. Thus, this case is not possible.

$\bullet$ or $l_{C_0}$ is fully included in the half-space $\{y<0\}$, that is, also in $\{Y<0\}$, and does not connect to $P_1$, as $C_0\not\in\mathcal{C}$. Moreover, Lemma \ref{lem.Q1} together with the fact that $p\geq p_F(\sigma)$ and the linearization \eqref{intermXX} imply that all the orbits $l_C$ with $C>0$ on the unstable manifold of $Q_1$ go out in the region
$$
y(\eta_1)>-\frac{x(\eta_1)}{N}\geq-\frac{\beta x(\eta_1)}{\alpha}, \qquad {\rm that \ is} \qquad Y=\frac{y}{x}>-\frac{\beta}{\alpha},
$$
entering the strip \eqref{strip}. Lemma \ref{lem.barlarge} and the fact that $C_0\not\in\mathcal{C}$ give that $l_{C_0}$ has to end up on the center-stable manifold of the critical point $P_{\gamma_0}$. The local behaviors \eqref{decay.Pg}, \eqref{decay.P1} and Lemma \ref{lem.monot} ordering decreasing profiles (corresponding to trajectories fully contained in the region $\{Y<0\}$) show that $\mathcal{A}=(C^*,\infty)$, $\mathcal{C}=(0,C_*)$ and $\mathcal{B}=[C_*,C^*]$, for some $0<C_*\leq C^*<\infty$ (eventually relabeled), and the similar classification for profiles $f(\cdot;A)$ follows from \eqref{bij}, completing the proof of both Theorems \ref{th.gen} and \ref{th.large}.
\end{proof}
We plot in Figure \ref{fig1} below the typical behavior of several trajectories $l_C$ with $C>0$, for several $C\in\mathcal{A}$ and several $C\in\mathcal{C}$. We have also plotted the plane $\{y=-(\sigma+2)/(p-m)\}$ and we see how all trajectories either go directly towards it, or cross it and then go back, in order to reach the critical point $P_1$, according to \eqref{P1xyz} below. This will be in strong contrast with the range $m<p<p_F(\sigma)$ considered in the next section.

\begin{figure}[ht!]
  % Requires \usepackage{graphicx}
  \begin{center}
  \includegraphics[width=11cm,height=7.5cm]{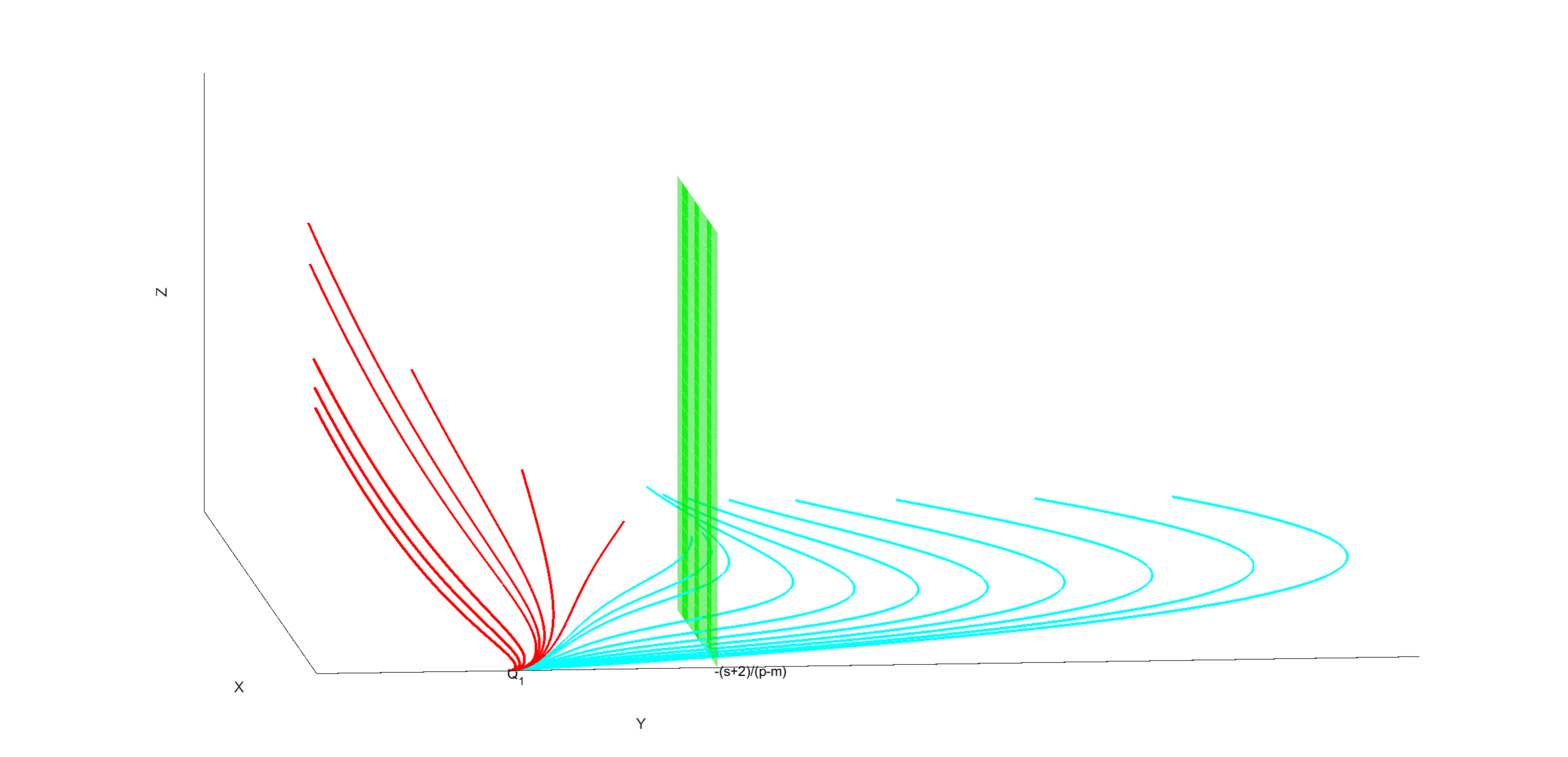}
  \end{center}
  \caption{Trajectories $l_C$ going out of $Q_1$. Numerical experiment for $m=3$, $N=3$, $\sigma=2.5$ and $p=5$.}\label{fig1}
\end{figure}

\subsection{Proof of Theorems \ref{th.gen} and \ref{th.small} for $m<p<p_F(\sigma)$: existence}\label{sec.small}

Let us fix $p\in(m,p_F(\sigma))$ throughout this section. We consider the plane $\{y=-(\sigma+2)/(p-m)\}$, which is a kind of ``filter" for the trajectories of the system \eqref{PSsyst2}. We thus prove a preparatory lemma.
\begin{lemma}\label{lem.plane}
A trajectory $(x,y,z)(\eta_1)$ of the system \eqref{PSsyst2} crossing the plane $\{y=-(\sigma+2)/(p-m)\}$ and entering the half-space $\{y<-(\sigma+2)/(p-m)\}$ cannot return afterwards to $\{y>-(\sigma+2)/(p-m)\}$. Moreover, there are no trajectories $(x,y,z)(\eta_1)$ of the system \eqref{PSsyst2} tangent to the plane $\{y=-(\sigma+2)/(p-m)\}$.
\end{lemma}
\begin{proof}
The flow of the system \eqref{PSsyst2} across the plane $\{y=-(\sigma+2)/(p-m)\}$ (with normal $(0,1,0)$) is given by the sign of the expression
$F(z)=z-Z_0$, where $Z_0$ has been introduced in \eqref{Z0}, hence, if a trajectory crosses the plane coming from the half-space $\{y>-(\sigma+2)/(p-m)\}$, it is through a point with $z<Z_0$. But we infer from the third equation of the system \eqref{PSsyst2} that in the half-space $\{y<-(\sigma+2)/(p-m)\}$ the $z$ coordinate is decreasing, thus it cannot increase again up to a value $z>Z_0$ in order to return to the half-space $\{y>-(\sigma+2)/(p-m)\}$. For the second statement, assume for contradiction that there is a trajectory $(x,y,z)(\eta_1)$ and some $\eta_{1,*}\in\real$ such that
$$
y(\eta_{1,*})=-\frac{\sigma+2}{p-m}, \qquad \dot{y}(\eta_{1,*})=0, \qquad y''(\eta_{1,*})\geq0.
$$
It then follows that $z(\eta_{1,*})=Z_0$ and thus, by the uniqueness theorem, we infer that the trajectory has to coincide with the line $z=Z_0$ fully included in the plane $\{y=-(\sigma+2)/(p-m)\}$, leading to a contradiction.
\end{proof}
Let us recall now that the critical point $P_1$, according to \eqref{cmf} and \eqref{centerP1}, together with \eqref{PSchange.inf}, is seen in $(x,y,z)$ variables as the limit
\begin{equation}\label{P1xyz}
x(\eta_1)\to\infty, \quad y(\eta_1)\to-\frac{\sigma+2}{p-m}, \quad {\rm as} \ \eta_1\to\infty.
\end{equation}
With these preparations, we are in a position to prove the existence part of Theorem \ref{th.small}.
\begin{proof}[Proof of Theorem \ref{th.small}: existence]
Let $m<p<p_F(\sigma)$. In a first step we introduce the disjoint sets
\begin{equation}\label{sets.small}
\begin{split}
&\mathcal{A}:=\{C\in(0,\infty): {\rm there \ is} \ \eta_{1,*}\in\real, \ y(\eta_{1,*})>0 \ {\rm on \ the \ trajectory} \ l_C\},\\
&\mathcal{C}:=\left\{C\in(0,\infty): y(\eta_1)<0 \ {\rm on} \ l_C \ {\rm for \ any} \ \eta_1\in\real \ {\rm and} \ \inf\limits_{\eta_1\in\real}y(\eta_1)\leq-\frac{\sigma+2}{p-m}\right\},\\
&\mathcal{B}:=(0,\infty)\setminus(\mathcal{A}\cup\mathcal{C}).
\end{split}
\end{equation}
Similarly as in the previous section, Lemmas \ref{lem.invar} and \ref{lem.x0} imply that there is $C^*>0$ such that $(C^*,\infty)\subseteq\mathcal{A}$. Let us now look at the set $\mathcal{C}$. According to Lemma \ref{lem.plane} and to \eqref{P1xyz}, for $C\in\mathcal{C}$, the trajectory $l_C$ has only two possibilities:

$\bullet$ either $l_C$ crosses the plane $\{y=-(\sigma+2)/(p-m)\}$ and then remains in that region forever. The set of parameters $C$ for which this occurs is an open set by definition.

$\bullet$ or $l_C$ stays in the half-space $\{y=-(\sigma+2)/(p-m)\}$ and enters the critical point $P_1$ in the limit $\eta_1\to\infty$, according to \eqref{P1xyz}. This is also an open set, according to the stability of $P_1$.

Altogether, $\mathcal{C}$ is an open set. Moreover, Lemma \ref{lem.z0} and the continuity with respect to $C$ give that there is $C_*>0$ such that $(0,C_*)\subseteq\mathcal{C}$. This also shows that $\mathcal{B}$ is non-empty and closed. Picking $C_0\in\mathcal{B}$, we infer from the fact that $C_0\not\in\mathcal{A}$, $C_0\not\in\mathcal{C}$ and the fact that the trajectory $l_{C_0}$ cannot be tangent to any of the planes $\{y=0\}$ or $\{y=-(\sigma+2)/(p-m)\}$ (as shown in the previous section, respectively in Lemma \ref{lem.plane}) that $l_{C_0}$ satisfies the condition \eqref{cond.region} and does not enter $P_1$. Lemma \ref{lem.region} then entails that $l_{C_0}$ connects to the critical point $P_{\gamma_0}$. The monotonicity of the decreasing profiles given in Lemma \ref{lem.monot} proves that $\mathcal{B}$ is a closed interval and completes the proof of Theorem \ref{th.gen} in this range.

\medskip

In a second step, we are left to show that there is some $C\in(0,\infty)$ such that $l_C$ connects to $P_2$ (corresponding then to the profile of the very singular, compactly supported self-similar solution as stated in Theorem \ref{th.small}). To this end, we split now the open set $\mathcal{C}$ defined in \eqref{sets.small}, at its turn, into three disjoint sets
\begin{equation}\label{sets.small2}
\begin{split}
&\mathcal{U}:=\{C\in\mathcal{C}: {\rm the \ trajectory} \ l_C \ {\rm connects \ to} \ Q_5\},\\
&\mathcal{V}:=\{C\in\mathcal{C}: {\rm the \ trajectory} \ l_C \ {\rm connects \ to} \ P_1\},\\
&\mathcal{W}:=\mathcal{C}\setminus(\mathcal{U}\cup\mathcal{V}).
\end{split}
\end{equation}
Once more, the non-emptyness of $\mathcal{U}$ follows from Lemma \ref{lem.z0}, while the stability of both $Q_5$ and $P_1$ shows that $\mathcal{U}$ and $\mathcal{V}$ are open sets. It remains to prove that $\mathcal{V}$ is non-empty. In order to show it, let
$$
\mathcal{S}:=\left\{C\in\mathcal{C}: {\rm the \ trajectory} \ l_C \ {\rm crosses \ the \ plane} \ y=-\frac{\sigma+2}{p-m}\right\}.
$$
It is obvious that $\mathcal{S}$ is an open set. Pick $C_0=\sup\mathcal{S}\in(0,\infty)$. Since $\mathcal{S}$ is open, we deduce that $C_0\not\in\mathcal{S}$, that is, the trajectory $l_{C_0}$ remains forever in the half-space $\{y>-(\sigma+2)/(p-m)\}$, since Lemma \ref{lem.plane} shows that there cannot be $\eta_1\in\real$ such that $y(\eta_1)=-(\sigma+2)/(p-m)$ without crossing the plane immediately after. Moreover, by the definition of supremum, there is a sequence $(C_{n})_{n\geq1}$ such that $C_n\in\mathcal{S}$, $C_n<C_0$ for every $n\geq1$ and $C_n\to C_0$ as $n\to\infty$. By the definition of $\mathcal{S}$, for every $n\geq1$ there is $\eta_{1,n}\in\real$ such that the point
$$
\left(x(\eta_{1,n}),-\frac{\sigma+2}{p-m},z(\eta_{1,n})\right) \ {\rm belongs \ to} \ l_{C_n}, \quad 0<z(\eta_{1,n})<Z_0,
$$
and is the crossing point between $l_{C_n}$ and the plane $\{y=-(\sigma+2)/(p-m)\}$. Assume for contradiction that $(x(\eta_{1,n}))_{n\geq1}$ is bounded. Since $(z(\eta_{1,n}))_{n\geq1}$ is also bounded, by extracting a subsequence (relabeled also $\eta_{1,n}$), we may assume that both $x(\eta_{1,n})$ and $z(\eta_{1,n})$ are convergent as $n\to\infty$ to some limits $x_{\infty}$ and $z_{\infty}$. By continuity with respect to $C$, we conclude that the point
$$
P_{\infty}:=\left(x_{\infty},-\frac{\sigma+2}{p-m},z_{\infty}\right)
$$
either belongs to $l_{C_0}$ or is the limit as $\eta_{1}\to\infty$ of $l_{C_0}$, in both cases reaching a contradiction. Indeed, the former contradicts Lemma \ref{lem.plane}, while the latter would imply that $P_{\infty}$ is a finite critical point for the system \eqref{PSsyst2}, and there is no such point. Observe that in fact, this contradiction shows that $(x_{\eta_{1,n}})_{n\geq1}$ has no convergent subsequences. Thus, $x(\eta_{1,n})\to\infty$ as $n\to\infty$, and the continuity with respect to $C$, \eqref{P1xyz}, and the stability of $P_1$ then give that $l_{C_0}$ connects to $P_1$, that is, $C_0\in\mathcal{V}$.

\medskip

We thus conclude that $\mathcal{U}$ and $\mathcal{V}$ are non-empty and open, and thus $\mathcal{W}$ is non-empty (and closed). Pick now $C_1\in\mathcal{W}$. In particular, since $C_1\in\mathcal{C}$ but $C_1\not\in\mathcal{V}$, it follows from Lemma \ref{lem.region} that the trajectory $l_{C_1}$ crosses the plane $\{y=-(\sigma+2)/(p-m)\}$. We further infer from Lemma \ref{lem.plane} that there exists $\eta_{1,*}\in\real$ such that at points $(x,y,z)(\eta_1)$ belonging to $l_{C_1}$ with $\eta_1>\eta_{1,*}$, we have $y(\eta_1)<-(\sigma+2)/(p-m)$. We deduce then from the third equation of the system \eqref{PSsyst2} that $z(\eta_1)<Z_0$ for any $\eta_1>\eta_{1,*}$ and $z(\eta_1)$ decreases on $(\eta_{1,*},\infty)$. Since $\eta_1\mapsto x(\eta_1)$ is increasing, it is easy to show by an argument by contradiction that $x(\eta_1)\to\infty$ as $\eta_1\to\infty$. Moreover, since $C_1\not\in\mathcal{U}$, owing to the fact that $Q_5$ is an attractor characterized by the limit $y/x\to-\infty$, we deduce that the function $\eta_1\mapsto y(\eta_1)/x(\eta_1)$, $\eta_1\in(\eta_{1,*},\infty)$, is uniformly bounded from below by some negative constant $-\kappa$ with $\kappa>0$ sufficiently large.

We change now the viewpoint of the trajectory $l_{C_1}$ and move to the $(X,Y,Z)$ variables by undoing \eqref{PSchange.inf}. The previous arguments imply that, on the trajectory $l_{C_1}$, we have
\begin{equation}\label{interm12}
X=\frac{1}{x}\to0, \quad Z=\frac{z}{x}\to0, \quad Y=\frac{y}{x}\in(-\kappa,0),
\end{equation}
as $\eta\to\infty$. If we assume for contradiction that $Y(\eta)$ is not convergent, but bounded, it follows that it has infinitely many oscillations, so that, there are sequences $(\eta_{k})_{k\geq1}$ and $(\eta_{j})_{j\geq1}$ of respectively minima and maxima of $Y(\eta)$, such that $\eta_{j}\to\infty$ and $\eta_{k}\to\infty$. Since $\dot{Y}(\eta_{j})=\dot{Y}(\eta_{k})=0$, we infer from the second equation in the system \eqref{PSsyst1} evaluated at $\eta_{k}$, respectively $\eta_{j}$, and \eqref{interm12} that
$$
-Y^2(\eta_j)-\frac{\beta}{\alpha}Y(\eta_j)\to0, \quad -Y^2(\eta_k)-\frac{\beta}{\alpha}Y(\eta_k)\to0,
$$
which together with the boundedness, readily gives that $Y(\eta_j)$ and $Y(\eta_k)$ may have only zero or $-\beta/\alpha$ as limit points. Since zero as limit point is excluded by the asymptotic stability of $P_1$ and the fact that $C_1\not\in\mathcal{V}$, we are left with $Y(\eta_j)\to-\beta/\alpha$ as $j\to\infty$ and $Y(\eta_k)\to-\beta/\alpha$ as $k\to\infty$. We conclude that $l_{C_1}$ ends up at the critical point $P_2$, and the existence is proved.
\end{proof}
We plot in Figure \ref{fig2} a bunch of trajectories for several values of $C\in(0,\infty)$, illustrating by a numerical experiment the previous alternative with several different behaviors and the ``filter" realized by the plane $\{y=-(\sigma+2)/(p-m)\}$.

\begin{figure}[ht!]
  % Requires \usepackage{graphicx}
  \begin{center}
  \includegraphics[width=11cm,height=7.5cm]{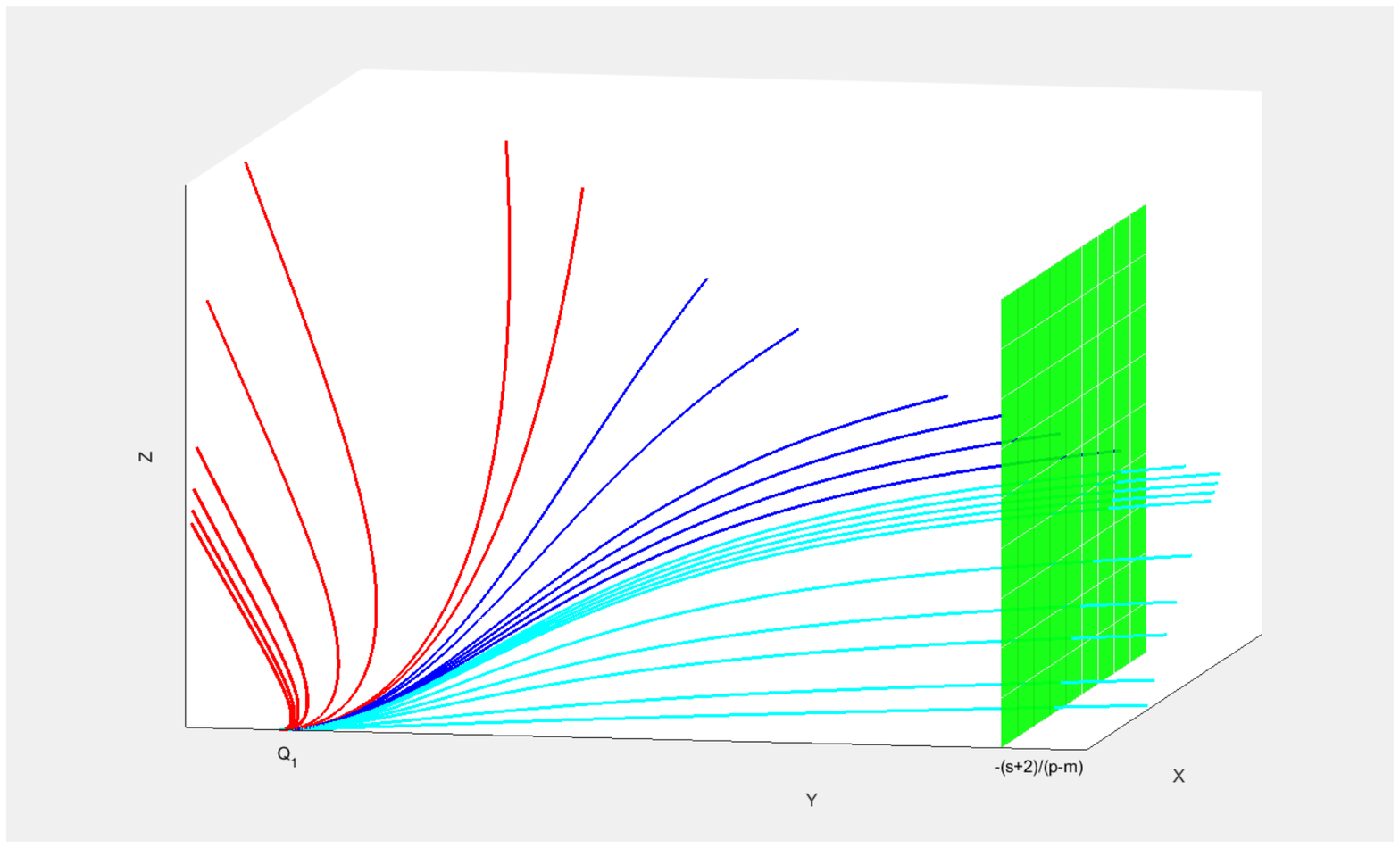}
  \end{center}
  \caption{Trajectories $l_C$ going out of $Q_1$. Numerical experiment for $m=3$, $N=3$, $\sigma=2.5$ and $p=3.4$.}\label{fig2}
\end{figure}

\subsection{Proof of Theorem \ref{th.small}: Uniqueness of the compactly supported profile}

In this section we prove that there exists only one compactly supported profile, as claimed in Theorem \ref{th.small}, part (a). The proof is done directly working with profiles $f(\cdot;A)$ with $A>0$, and follows rather closely the analogous one in \cite[Section 5]{IMS23}. As done there, we begin with a preparatory result.
\begin{lemma}\label{lem.super}
Let $f$ be a solution to \eqref{SSODE} and define
\begin{equation}\label{super}
U_{\lambda}(x,t):=t^{-\alpha}f_{\lambda}(|x|t^{-\beta}), \qquad f_{\lambda}(\xi)=\lambda^{-2/(m-1)}f(\lambda\xi),
\end{equation}
for $\lambda\in(0,1)$, where $\alpha$, $\beta$ are given in \eqref{SSexp}. Then $U_{\lambda}$ is a supersolution to Eq. \eqref{eq1}.
\end{lemma}
\begin{proof}
It follows by direct calculation. Indeed, we have
\begin{equation*}
\begin{split}
U_{\lambda,t}-\Delta U_{\lambda}^m-|x|^{\sigma}U_{\lambda}^p&=t^{-\alpha-1}\left[-\alpha f_{\lambda}-\beta\xi f_{\lambda}'-(f_{\lambda}^m)''-\frac{N-1}{\xi}(f_{\lambda}^m)'+\xi^{\sigma}f_{\lambda}^p\right]\\
&=t^{-\alpha-1}\xi^{\sigma}f_{\lambda}^p\left(1-\lambda^{L/(m-1)}\right)>0.
\end{split}
\end{equation*}
\end{proof}
We are now in a position to complete the proof of part (a) of Theorem \ref{th.small}.
\begin{proof}[Proof of Theorem \ref{th.small}, part (a)]
Assume for contradiction that there exist two profiles $f_1=f(\cdot;A_1)$ and $f_2=f(\cdot;A_2)$ with $A_1<A_2$ and with interfaces (in the sense described in the statement of Theorem \ref{th.small}, part (a)) at finite points $\xi_1=\xi_{\rm max}(A_1)$, respectively $\xi_2=\xi_{\rm max}(A_2)\in(0,\infty)$. We infer from Lemma \ref{lem.monot} that $\xi_1<\xi_2$ and $f_1(\xi)<f_2(\xi)$ for any $\xi\in[0,\xi_1]$. Consider now the same rescaling introduced in \eqref{resc} and define $\lambda_0\in(0,1)$ to be the optimal parameter as defined in \eqref{interm10}. We already know from the proof of Lemma \ref{lem.monot} that $g_{\lambda_0}$ and $g_2$, thus also $f_{\lambda_0}$ and $f_2$, remain strictly ordered in $[0,\xi_2)$. The optimality of $\lambda_0$ thus entails that the contact point between $f_{\lambda_0}$ and $f_2$ has to lie at their common edge of the support $\xi=\xi_2$. We thus have
$$
0=f_2(\xi_2)=f_{\lambda_0}(\xi_2), \qquad 0<f_2(\xi)<f_{\lambda_0}(\xi) \ {\rm for \ any} \ \xi\in[0,\xi_2).
$$
We go back to the time variable and construct the self-similar functions
$$
U_2(x,t)=t^{-\alpha}f_2(|x|t^{-\beta}), \qquad U_{\lambda_0}=t^{-\alpha}f_{\lambda_0}(|x|t^{-\beta}),
$$
thus $U_2$ is a solution to Eq. \eqref{eq1} and $U_{\lambda_0}$ is a supersolution to Eq. \eqref{eq1} by Lemma \ref{lem.super}. We next apply an idea of \emph{separation of supports} similar to the one used in \cite{IMS23}. More precisely, since at $t=1$ we have a perfect identification between the function and its profile, we notice that
$$
U_2(x,1)\leq U_{\lambda_0}(x,1)<U_{\lambda_0}(x,1+\delta),
$$
provided we take a small $\delta>0$. Indeed, contrasting with the proof in \cite[Section 5]{IMS23}, in our case a small time advance means a small advance of the support but a small decrease in amplitude. Thus, we have to ensure that we still keep the order at zero (since a tangency at an interior point cannot appear, in the same way as shown in the proof of Lemma \ref{lem.monot}), that is,
$$
U_2(0,1)=A_2<(1+\delta)^{-\alpha}\lambda_0^{-2/(m-1)}A_1=U_{\lambda_0}(0,1+\delta),
$$
which allows us to choose a $\delta>0$ sufficiently small, since by the strict ordering we already know that $A_2<\lambda_0^{-2/(m-1)}A_1$. With this choice of $\delta$, we find that $U_2(x,1)$ and $U_{\lambda_0}(x,1+\delta)$ are strictly separated on the compact interval $[0,\xi_2]$. There exists thus, by continuity with respect to $\lambda$, a better parameter $\lambda_1\in(\lambda_0,1)$ such that
$$
U_2(x,1)<\lambda_1^{-2/(m-1)}(1+\delta)^{-\alpha}f_1(\lambda_1|x|(1+\delta)^{-\beta})=U_{\lambda_1}(x,1+\delta).
$$
Since $U_2$ is a solution and $U_{\lambda_1}$ is a supersolution to Eq. \eqref{eq1} (according to Lemma \ref{lem.super}), we deduce that $U_2(x,t)<U_{\lambda_1}(x,t+\delta)$ for any $t>1$. This follows from the comparison principle (which is a standard property for absorption-diffusion equations, and will also be proved for completeness in the companion paper to this work), but since we are dealing with functions having the specific self-similar form, we can actually show that no contact point between $U_2(t)$ and $U_{\lambda_1}(t+\delta)$ may appear at a first later time $t_1>1$ by removing the contact points in the same way as we did in the proof of Lemma \ref{lem.monot}. Recalling the expressions of $U_2(x,t)$ and $U_{\lambda_1}(x,t+\delta)$, we are left with
\begin{equation}\label{interm14}
f_2(\xi)\leq\left(\frac{t+\delta}{t}\right)^{-\alpha}\lambda_1^{-2/(m-1)}f_1\left(\lambda_1\xi\left(\frac{t+\delta}{t}\right)^{-\beta}\right),
\end{equation}
for any $t>1$. By letting $t\to\infty$ in the right hand side of \eqref{interm14}, we deduce that $f_2(\xi)\leq f_{\lambda_1}(\xi)$ for any $\xi\in[0,\xi_2]$, contradicting the optimality of $\lambda_0$ in \eqref{interm10}. This contradiction implies the uniqueness of the compactly supported profile, ending the proof.
\end{proof}

\noindent \textbf{Acknowledgements} R. I. is partially supported by the Spanish project PID2020-115273GB-I00. The authors wish to thank Ariel S\'anchez (Univ. Rey Juan Carlos, Madrid) for interesting comments and support.

\bigskip

\noindent \textbf{Data availability} Our manuscript has no associated data.

\bigskip

\noindent \textbf{Conflict of interest} The authors declare that there is no conflict of interest.

\bibliographystyle{plain}

\end{document}